\documentclass{amsart}
\usepackage{graphicx}
\usepackage{amssymb}
\usepackage{amsfonts}
\swapnumbers
\newtheorem{theorem}{Theorem}[section]
\newtheorem{lemma}[theorem]{Lemma}
\newtheorem{corollary}[theorem]{Corollary}
\newtheorem{proposition}[theorem]{Proposition}
\theoremstyle{definition}

\newtheorem{remark}[theorem]{Remark}
\newtheorem{example}[theorem]{Example}

\numberwithin{equation}{section}
\theoremstyle{plain}    

\numberwithin{equation}{section} 
\numberwithin{figure}{section} 
\theoremstyle{plain}    
\theoremstyle{plain}    
\theoremstyle{remark}    
\newtheorem*{acknowledgement*}{Acknowledgement} 

\newcommand{\cD}{{\mathcal D}}

\newcommand{\cL}{{\mathcal L}}

\newcommand{\cS}{{\mathcal S}}
\newcommand{\cT}{{\mathcal T}}

\newcommand{\cW}{{\mathcal W}}

\newcommand{\cY}{{\mathcal Y}}
\newcommand{\cZ}{{\mathcal Z}}

\newcommand{\te}{{\theta}}
\newcommand{\Te}{{\Theta}}

\newcommand{\ve}{{\varepsilon}}
\newcommand{\del}{{\delta}}

\newcommand{\gam}{{\gamma}}
\newcommand{\Gam}{{\Gamma}}

\newcommand{\sig}{{\sigma}}
\newcommand{\al}{{\alpha}}

\newcommand{\Up}{{\Upsilon}}

\newcommand{\bbR}{{\mathbb R}}

\newcommand{\bbZ}{{\mathbb Z}}
\newcommand{\bbI}{{\mathbb I}}

\newcommand{\bfP}{{\mathbf P}}

\newcommand{\bfZ}{{\mathbf Z}}

\usepackage{color}

\begin{document}
\title[]{Tails of polynomials of random variables and stable limits 
for nonconventional sums}%
\vskip 0.1cm 
\author{ Yuri Kifer
\quad\quad and\quad\quad\quad\quad S.R.S. Varadhan\\
\vskip 0.1cm
Institute of Mathematics\quad\quad\quad Courant Institute\\
The Hebrew University\quad\quad\quad New York University\\
Jerusalem, Israel\quad\quad\quad\quad\quad\quad New York, USA}%
\address{
Institute of Mathematics, The Hebrew University, Jerusalem 91904, Israel}
\email{ kifer@math.huji.ac.il}
\address{
Courant Institute for Mathematical Studies, New York University, 
251 Mercer St, New York, NY 10012, USA}%
\email{ varadhan@cims.nyu.edu}%

\thanks{Yu. Kifer was supported by ISF grant 82/10 and Einstein Foundation,
Berlin grant A 2012 137, S.R.S. Varadhan was supported by NSF grant DMS  
1208334}
\subjclass[2000]{Primary: 60F17 Secondary: 60E07, 60G52}%
\keywords{stable distributions, heavy tails, limit theorems, Levi process, 
nonconventional
 sums.}%
\dedicatory{  }
\date{\today}
\begin{abstract}\noindent
First, we obtain decay rates of probabilities of tails of 
polynomials in  several  independent random variables with heavy tails. 
Then we derive stable limit theorems for  sums of the form
 $\sum_{Nt\geq n\geq 1}F\big(X_{q_1(n)},\ldots,X_{q_\ell(n)}\big)$ where $F$ 
 is a polynomial, $q_i(n)$ is either $n-1+i$ or $ni$  and $X_n,n\geq 0$ is
 a sequence of independent  identically distributed random variables with heavy
  tails.  Our results can be viewed as an extension to the heavy tails
 case of the nonconventional functional central limit theorem from \cite{KV1}.

\end{abstract}
\maketitle
\markboth{Yu.Kifer and S.R.S.Varadhan}{Nonconventional stable limit theorems} 
\renewcommand{\theequation}{\arabic{section}.\arabic{equation}}
\pagenumbering{arabic}

\renewcommand{\theequation}{\arabic{section}.\arabic{equation}}
\pagenumbering{arabic}

\section{Introduction}\label{sec1}\setcounter{equation}{0}

Nonconventional ergodic theorems (with the name coming from \cite{Fur}) 
motivated originally by multiple recurrence problems have attracted much 
attention  during the last 30 years. Probabilistic limit theorems for
corresponding expressions have appeared more recently, in particular,
a functional central limit theorem for nonconventional sums of the form
\begin{equation}\label{1.1}
S_N(t)=\sum_{Nt\geq n\geq 1}F\big(X_{q_1(n)},\ldots,X_{q_\ell(n)})
\end{equation}
was obtained in \cite{KV1} for sequences of random variables $X_1,X_2,...$
with weak dependence. In this paper we consider polynomial functions
\begin{equation}\label{1.2}
F(x_1,...,x_\ell)=\sum_{\sigma_1,\ldots,\sigma_\ell}h_{\sigma_1\cdots 
\sigma_\ell} \,x_1^{\sigma_1}\cdots x_\ell^{\sigma_\ell},
\end{equation}
where the sum is taken over a finite set of nonnegative integer indexes,
and study  the tail probabilities of $F(X_1,X_2,...,X_\ell)$ for independent
random variables $X_i,\, i\geq 1$ with heavy tails. Then we obtain various
 results concerning convergence of distributions of properly normalized and
 centralized sums of the form (\ref{1.1}) where for all $i$'s either
  $q_i(n)=n-1+i$ or $q_i(n)=in$, $F$ has the form (\ref{1.2}) and  
  $X_i$'s are independent and have identical distributions with heavy tails. 
  
  We start with random variables $X$ having the same asymptotical
tail behavior as stable distributions, i.e.
\begin{equation}\label{1.3}
\bfP\{\pm X>x\}\sim c^\pm x^{-\al}\,\,\mbox{as}\,\, x\to\infty
\end{equation}
where $0<\al <2$ and $c^\pm\geq 0$ with $c^++c^->0$. Then we consider the
minimal class which contains all such random variables and which is closed
with respect to sums and products of independent random variables with such
tails. It turns out that this class consists of random variables $X$ with 
tails
\begin{equation}\label{1.4}
P\{\pm X>x\}\sim c^\pm x^{-\al}(\ln x)^k\,\,\mbox{as}\,\, x\to\infty
\end{equation}
where $k\geq 0$ is an integer. 

Namely, if we start with $\ell$ independent random variables $X_1,...,X_\ell$
with tails of the form
\begin{equation}\label{1.5}
P\{\pm X_i>x\}\sim c_i^\pm x^{-\al_i}(\ln x)^{k_i}\,\,\mbox{as}\,\, x\to
\infty,
\end{equation}
where $0<\al<2$, $k_i\geq 0$ are some integers and $c^\pm\ge 0$, $c^+_i+c^-_i>0$,
then the polynomial $Y=F(X_1,...,X_\ell)$, with $F$ given by (\ref{1.2}), will 
have the tail behavior
\begin{equation}\label{1.6}
P\{\pm Y>x\}\sim  c_\ast^\pm x^{-\al_\ast}(\ln x)^{k_\ast}\,\,\mbox{as}\,\,
 x\to\infty
\end{equation}
where $c_\ast^+ +c_\ast^->0$, $0<\al_\ast <2 $ and both $\alpha_*$ and an 
integer $k_\ast\geq 0$  can be explicitly described. Actually, this holds
true for a somewhat larger class of random variables $X_i$ having tail
distributions satisfying (\ref{1.5}) with any real  $k_i\geq 0$, and so we
obtain our results for the latter. Allowing 
negative integer values for $k$ will generate tail behavior that contains 
additional $\ln\ln x$ terms. In fact, our method works directly also for
$F$ given by a finite sum (\ref{1.2}) with arbitrary real nonnegative
$\sig_1,...,\sig_\ell$ provided random variables $X_i$'s are nonnegative
to make sense of their arbitrary real powers.

Next, we consider nonconventional sums (\ref{1.1}) where $F$ is a polynomial
of the form (\ref{1.2}), \hbox{$\{ X_i,\, i\geq 1\}$ } is a sequence of 
independent identically distributed (i.i.d.) random variables with the tail 
behavior given by (\ref{1.4}). Concerning integer valued functions 
$1\leq q_1(n)<q_2(n)<\cdots <q_\ell(n)$ we will concentrate on two important
 cases, namely, the $\ell$-dependence case $q_j(n)=n+j-1$ and the long range 
 dependence
arithmetic progression case  $q_j(n)=jn$ which leads to an interesting 
comparison
of  results for these two cases. After obtaining the tail behavior of the 
form
(\ref{1.6}) for summands of $S_N$ we proceed to establishing limit theorems 
for
$S_N$ showing that for a particular class of polynomials $F$ and certain
sequences $a_N$ and $b_N$,
\[
\frac 1{b_N}(S_N(t)-Nta_N)
\]
weakly converges in the Skorokhod $J_1$ topology to an $\al_\ast$-stable 
L\' evy
process. On the other hand, for general polynomials $F$  in the 
$\ell$-dependence case $q_i(n)=n+j-1$ the convergence in $J_1$ may not
 hold true though finite dimensional distributions always weakly converge.
 In the arithmetic progression case $q_j(n)=jn$,  while the convergence in 
 $J_1$ takes place for any polynomial $F$, the limiting process in some cases 
 may have dependent increments.

In the $\ell$-dependence case the summands form a stationary sequence, and so
after verifying corresponding conditions it is possible to rely on previous
results (see \cite{TK}). On the other hand, in the arithmetic progression case $q_j(n)=jn$
the sum $S_N$ consists of  summands  with a strong long range
dependence which do not form a stationary sequence. This does not enable
us to rely directly on existing results concerning stable limit theorems for
 sums of stationary weakly dependent random variables 
(see, for instance \cite{DR}, \cite{Da}, \cite{DJ}, \cite{TK} and references
 there). Some of our results can be extended to a wider class of 
random variables which have tail behavior given by (\ref{1.4}) with more 
general slowly varying  functions in place of powers of the 
 logarithm but then most of our explicit computations will not be available.

 \section{Preliminaries and main results}\label{sec2}\setcounter{equation}{0}

Let $X_1,...,X_\ell$ be independent random variables with tail distribution
satisfying
\begin{equation}\label{2.1}
\lim_{z\to\infty}z^{\al_j}(\ln z)^{-k_j}P\{\pm X_j>z\}=c_j^\pm,\, 
j=1,...,\ell
\end{equation} 
for some $\al_j\in(0,2),\, k_j\geq 0$ and $c_j^\pm\geq 0$ with
$c_j^++c_j^->0$, $j=1,...,\ell$. We will often consider a space $\mathcal W$
of bounded continuous functions $W$ defined on some $m$-dimensional
Euclidean space $R^m$ and satisfying
\begin{equation}\label{2.1+}
|W(x)|\le C|x|^2(1+|x|^2)^{-1}
\end{equation}
for some constant $C>0$. It is known (see Theorem 3.6 of \cite{Re} and 
Lemma \ref{l5.1} in Appendix) that (\ref{2.1}) holds true if and only if
\begin{equation}\label{2.2}
\lim_{z\to\infty}z^{\al_j}(\ln z)^{-k_j}E[ W(\frac{X_j}{z})]=\int W(x)
(c^j_-{\bf 1}_{x<0}+c^j_+{\bf 1}_{x<0})\frac{dx}{|x|^{1+\alpha_j}}
\end{equation} 
for all $W\in\mathcal W$ (here $m=1$) and $j=1,\ldots,\ell$.

 First, we want to establish the tail behavior of random variables
  $Y=F(X_1,...,X_\ell)$  where $F$ is a  polynomial of the form
\begin{equation}\label{2.2+}
F(x_1,...,x_\ell)=\sum_{\theta\in\Theta} h_\theta g_\theta(\bf x)  
\end{equation}
where $\theta=(\sigma_1,\ldots,\sigma_\ell)$  is a multi index from a finite 
set $\Te$ of multi indices of integers $\sigma_i> 0$ and $g_\theta(\bf x)=
x_1^{\sigma_1}\cdots x_\ell^{\sigma_\ell}$ is a monomial.
We will also assume that $\Theta$  consists only of those $\theta$ for which 
$h_\theta\neq 0$.

Introduce the following notations
\begin{eqnarray*}
&\alpha(\theta)=\min_{1\leq j\leq\ell}\frac {\al_j}{\sigma_j},
\, J(\theta)=\{ j: \frac{\al_j}{\sigma_j}=\alpha(\theta) \},\, p(\theta)=
|J(\theta)|=\#\{ j: \frac{\al_j}{\sigma_j}=\alpha(\theta)\}\nonumber\\
&k(\theta)=p(\theta)-1+\sum_{j\in J(\theta)} k_j,\,\alpha_*=\min_{\theta\in
\Theta} \alpha(\theta),\, k_*=\max_{\theta\in \Theta, \alpha(\theta)=
\alpha_*} k(\theta)\nonumber\\
&\mbox{and}\,\,\Theta_*=\{ \theta:\,\alpha(\theta)=\alpha_* \,\,\mbox {and} 
\, \,k(\theta)=k_*\}.\nonumber
\end{eqnarray*}
Now we can state our first main result. Consider the collection of random 
variables $\{g_\theta(X_1,\ldots, X_\ell)\}$
and view them as a random vector $\bf Z$ in $R^m$ where $m$ is the cardinality 
of $\Theta$. 

\begin{theorem}\label{thm2.1} (i) The limit
\begin{equation}\label{2.3}
\lim_{z\to\infty} z^{\alpha_*} (\ln z)^{-k_*} E[W(\frac{\bf Z}{z})]=
\int_0^\infty \int_{S^{m-1}}W(rs)\frac{\nu(ds)\, dr}{r^{1+\alpha_*}}
\end{equation}
exists for any $W\in\mathcal W$ with some measure $\nu$ on the sphere 
$S^{m-1}$
satisfying $\nu(S^{m-1})>0$. Moreover, let $m_*$ be the cardinality of 
$\Theta_*$ and
$R^{m_*}$ be any $m_*$-dimensional subspace containing the support of the 
distribution
of the random vector $\{ g^*_\theta(X_1,...,X_\ell),\,\theta\in\Theta\}$ where
$g^*_\theta=g_\theta$ if $\theta\in\Theta_*$ and $g_\theta^*\equiv 0$ if 
$\theta\in\Theta\setminus\Theta_*$.
Then $\nu\big(S^{m-1}\setminus (R^{m_*}\cap S^{m-1})\big)=0$.

(ii) The tail behavior of the polynomial $Y=F(X_1,...,X_\ell)$ is given by
\begin{equation}\label{2.4}
\lim_{z\to\infty} z^{\al_*}(\ln z)^{-k_*}P\{\pm Y>z\}=c^F_\pm
\end{equation}
where $c_F^\pm\geq 0$ satisfies $c_F^++c_F^->0$. 
\end{theorem}

Theorem \ref{thm2.1} claims, in particular, that the class of random variables
satisfying (\ref{2.1}) is closed with respect to taking products and sums of 
products of independent random variables. We will prove Theorem \ref{thm2.1}
by establishing first the joint tail behavior of the collection of monomials 
$\{g_\theta (X_1,...,X_\ell),\,\theta\in\Theta\}$ which is the assertion (i)
and deduce from it the tail behavior  of the linear combination    
$\sum_{\theta\in\Theta} h_\theta g_\theta({X_1,...,X_\ell})$ which is the
assertion (ii) of the theorem. 

It is natural to inquire whether the wider class of random variables $X_j$ 
satisfying (\ref{2.1}) with arbitrary integers $k_j$ is closed with respect
to products of independent random variables, as well. This turns out to be
false as the following example shows.
\begin{example}\label{ex2.2}
Let $X_1$ and $X_2$ be independent symmetric random variables such that
\begin{equation}\label{2.5}
\lim_{z\to\infty}z^\al P\{ X_1>z\}=c_1\,\,\mbox{and}\,\,
\lim_{z\to\infty}z^\al(\ln z)P\{ X_2>z\}=c_2.
\end{equation}
Then
\begin{equation}\label{2.6}
\lim_{z\to\infty}z^\al(\ln\ln z)^{-1}P\{ X_1X_2>z\}=2c_1c_2\al.
\end{equation}
\end{example}
We will make necessary computations leading to (\ref{2.6}) at the end of 
Section \ref{sec3}. This example also shows that it may be difficult to
obtain general precise folmulas for tail asymptotics beyond the class of
random variables satisfying (\ref{2.1}).

Next, we consider a sequence of i.i.d. random variables $X_1,X_2,...$
satisfying  
\begin{equation}\label{2.1a}
\lim_{z\to\infty}z^{\al}(\ln z)^{-k} P\{\pm X_j>z\}=c_\pm
\end{equation}
where, again, $\al\in(0,2),\, k\geq 0,\, c^\pm\geq 0$ and $c^++c^->0$.
For $0\le t\le T$, set
\begin{equation}\label{2.10}
S_N(\theta,t)=\sum_{1\le n\le Nt} g_\theta(X_{q_1(n)},X_{q_2(n)}
,...,X_{q_\ell(n)})
\end{equation}
and 
\begin{equation}\label{2.7}
S_N(t)=\sum_{1\le n\le Nt} F(X_{q_1(n)},X_{q_2(n)},...,X_{q_\ell(n)})
=\sum_{\theta\in\Theta}h_\theta S_N(\theta, t).
\end{equation}
Here $F$ is the same as in (\ref{1.2}) and (\ref{2.2+}) while the integer
valued functions \hbox{$1\leq q_1(n)<q_2(n)<\cdots <q_\ell(n)$ } will be 
considered
here in two situations
\begin{equation}\label{2.8}
\ell-\mbox{dependence case: }\,\, q_j(n)=n+j-1,\,\, j=1,2,...,\ell;\,
n=1,2,...\,\,\mbox{and}
\end{equation}
\begin{equation}\label{2.9}
\mbox{arithmetic progression case: }\,\, q_j(n)=jn,\,\, j=1,2,...,\ell;\,
n=1,2,....
\end{equation}
Set  
\begin{eqnarray*}
&\sigma(\theta)=\max_{1\le j\le \ell}\sigma_j,\,\alpha(\theta)=\frac{\alpha}
{\sigma(\theta)}, \,J(\theta)=\{j: \sigma_j=\sigma(\theta)\};\,
 p(\theta)=\# \{j:j\in J(\theta)\},\\
 & k(\theta)=(k+1)p(\theta)-1\,\,\mbox{and}\,\, \alpha_*=\min_{\theta\in\Theta}
 \alpha(\theta),\, k_*=\max_{\theta: \alpha(\theta)=\alpha_*} k(\theta).
 \end{eqnarray*}
 We define again $\Te_*$ as above and observe that for all $\te\in\Te_*$,
 \[
 p_*=p(\te)=(k_*+1)(k+1)\quad\mbox{and}\quad\sig_*=\sig(\te)=\al\al_*^{-1}
 \]
 are the same. We denote also by $m_*$ the cardinality of $\Te_*$.

Consider the collection of random variables $\{Z_\theta,\,\theta\in\Theta\}$
 where $Z_\theta=g_\theta(X_1,\ldots,X_\ell)$. Let
$$
b_N=N^{\frac {1}{\al_*}}(\frac {1}{\al_*}\ln N)^{\frac{ k_*}
{\al_*}},\,\,\mbox{and}\,\,
a^\theta_N=E[\frac{b_N^2 Z_\theta}{ b_N^2+Z_\theta^2}]
$$
In Section \ref{sec5} we will establish the following limit theorem
for the $\ell$-dependence case.

\begin{theorem}\label{thm2.3} Let $q_j(n)$ be defined by (\ref{2.8}).

(i) As $N\to\infty$, all finite dimensional distributions of the $R^m$ valued
  process
\begin{equation}\label{2.12}
\Xi_N(\theta, t)=\frac {1}{b_N}\big(S_N(\theta,t)-Nta^\theta_N\big),\,\, 
t\in [0,T],\,\,\theta\in\Theta
\end{equation}
 converge weakly to the corresponding finite dimensional distributions of 
an $\al _*$-stable L\' evy process $\{\Xi(\theta,t),\, \theta\in\Theta\},\,
t\in[0,T]$ where $\Xi(\theta,\cdot)\equiv 0$ for $\theta\in\Theta\setminus
\Theta_*$.
 For each $\theta$ the process $\Xi_N(\theta, t)$ converges in $J_1$ 
topology but the $R^m$ valued vector process  $\{\Xi_N(\theta,\cdot),\,\theta
\in\Theta\}$
 may not converge weakly in any of Skorokhod's $J_1,J_2,M_1$ or $M_2$ 
 topologies.

(ii) Suppose now that in the representation (\ref{2.10}) there
exists no pair $\theta_1,\theta_2\in\Theta$ satisfying
 $J(\theta_2)=J(\theta_1)+r$ in the sense that there is 
no  integer $r$ such that  $i\in J(\theta_1)$ if and only if $i+r\in 
J(\theta_2)$. Then, as $N\to\infty$, the vector process $\{\Xi_N(\theta,\cdot),
\,\theta\in\Theta\}$ converges weakly in the $J_1$ topology on the space 
$D([0,T],\,\bbR^m)$
to the above $\al_*$-stable vector L\' evy process $\{\Xi(\theta,\cdot),\,
\theta\in\Theta\}$ and the component processes $\Xi(\theta,\cdot)$
are mutually independent for different $\te\in\Te$.

(iii) The sum $\xi_N(t)=\sum_{\theta\in\Theta}h_\theta \,\Xi_N(\theta, t)$ 
converges to the $\alpha_*$-stable L\' evy process 
$\xi(t)=\sum_{\theta\in\Theta}h_\theta\,\Xi(\theta, t)$, in general,
 only in the sense of weak convergence of finite dimensional distributions 
 while under the additional condition of (ii) the convergence is in the $J_1$ 
topology. 
\end{theorem}

For the arithmetic progression case we will obtain in Section \ref{sec6} the 
following result.

\begin{theorem}\label{thm2.4} Let $S_N(\theta,t)$ and $\Xi_N(\theta,t)$ be 
defined by (\ref{2.10}) and (\ref{2.12}), respectively, with $q_j(n)$ 
defined by (\ref{2.9}).

(i) As $N\to\infty$, the $R^m$-valued vector process $\{\Xi_N(\theta,\cdot),
\,\theta\in\Theta\}$ converges weakly in the $J_1$ topology on the space $D([0,T],\,\bbR^m)$
 to a process $\{\Xi(\theta,\cdot),\,\theta\in\Theta\}$ where 
 $\Xi(\theta,\cdot)
 \equiv 0$ for $\theta\in\Theta\setminus\Theta_*$. Each $\Xi(\theta,\cdot),\,
 \theta\in\Theta_*$ is an $\alpha_*$ stable L\' evy process but, in general, 
 these processes are mutually dependent. Their dependence structure will be
  clarified in the proof. Furthermore, the vector process $\{\Xi(\theta,\cdot),
  \,\theta\in\Theta\}$ has, in general, dependent increments.

(ii) Suppose that in the representation (\ref{2.10}) there exists no pair
 $\theta_1,\theta_2\in\Theta_*$ satisfying $J(\theta_2)=r\,J(\theta_1)$ in the 
 sense that there is no positive number $r$  such that $i\in J(\theta_1)$ if 
 and only if $ir\in J(\theta_2)$. Then, as $N\to\infty$, for each $\theta\in
 \Theta_*$ the  process $\Xi_N(\theta,\cdot)$ weakly converges in $J_1$ 
 topology to an  $\alpha_*$-stable L\' evy process $\Xi(\theta,\cdot)$ and 
 the latter processes are independent. Moreover, as $N\to\infty$ the vector 
 process $\{\Xi_N(\theta,\cdot),\,\te\in\Te_*\}$ weakly converges in the $J_1$
 topology of $D([0,T];R^{m_*})$ to the $m_*$-dimensional $\alpha_*$ stable 
 L\' evy process $\{\Xi(\theta,\cdot),\,\te\in\Te_*\}$.
 
 (iii) The sum $\xi_N(t)=\sum_{\theta\in\Theta}h_\theta\Xi_N(\theta,t ),\, 
 t\in[0,T]$
  converges weakly in the $J_1$ topology to an $\al_*$-stable process $\xi(t)
  =\sum_{\theta\in\Theta}h_\theta\,\Xi(\theta, t),
  \, t\in[0,T]$ which may have, in general, dependent increments but under the 
   additional condition of (ii) this process has independent increments, i.e.
    it is a L\' evy process. 
\end{theorem}

Setting $Y_n=F(X_{q_1(n)},X_{q_2(n)},...,X_{q_\ell(n)})$
we observe that when $q_j(n)=n+j-1$ then the sequence $\{Y_n,n\geq 1\}$ is
stationary and $\ell$-dependent, and  it is known that under conditions which 
can be veryfied in our circumstances the stable limit theorem holds true (see,
for instance, \cite{TK}). 
When $q_j(n)=jn$ the sequence $\{Y_n; \,n\geq 1\}$ is strongly long 
 range dependent and it is not stationary.  So we are not able to rely 
 directly  on any known results. We deal with this case
 establishing first a multidimensional stable limit theorem for 
 $\Xi_N(\theta, t)$ splitting the whole sum into independent subsums
 similarly to \cite{KV2} and applying some time rescaling. It turns out that
under additional arithmetic conditions specified in the above theorems
the limiting behavior as $N\to\infty$ of the process $\Xi_N$ is similar in
both cases (\ref{2.8}) and (\ref{2.9}) while, in general, it is quite 
different in these two cases. Indeed, in the $\ell$-dependence case consider 
$F(x_1,x_2)=
x_2-x_1$ with $\ell=2$ then $S_N(t)=X_{[Nt]}-X_1$ and it is not difficult 
to understand (see \cite{AT}) that all finite dimensional distributions
of $S_N/b_N$ converge to the unit mass at 0 while there is no weak 
convergence in any of Skorokhod's topologies. It is shown also in \cite{AT}
that if, for instance, we take here $F(x_1,x_2)=x_1+x_2$ then there 
will be weak convergence of $S_N/b_N$ in the $M_1$ topology but not in $J_1$.
On the other hand, in the arithmetic progression case with $F(x_1,x_2)=x_1+x_2$ 
and $\ell=2$ the weak convergence of $S_N/b_N$ in $J_1$ will hold true but the 
increments of the limiting process on the time intervals $[T/4,T/2]$ and
$[T/2,T]$ will be dependent. The same remains true for vector processes 
from Theorem \ref{thm2.4}(i) considering $g_{\te_1}(x_1,x_2)=x_1$ and
$g_{\te_2}(x_1,x_2)=x_2$ so that we will be dealing with the sums
$S_N(t)=\sum_{1\leq n\leq Nt}(X_n,X_{2n})$. Then the partial sums from $TN/4$
to $TN/2$ and from $TN/2$ to $TN$ will be strongly dependent which will lead
to dependence of increments of the limiting vector process on the time
intervals $[T/4,T/2)$ and $[T/2,T]$.

\begin{remark}\label{rem2.5}
The truncated average $E(Z_\theta\bbI_{|Z_{\theta}|\leq b_n})$ is often
taken as a centering expression in stable limit theorems in place of
$a^\theta_N$ introduced above. The latter is more convenient for our
purposes and this change does not influence convergence in the 
corresponding limit theorem but leads only to an additional drift term in
the L\' evy limiting process. Actually, we can 
interpret truncation in a wider sense as replacing $EX$ by $E(f(X))$ where
$f$ is a bounded function and $|f(x)-x|=o(|x|^2)$ near $0$. Two common 
choices are $f(x)=x{\bf 1}_{|x|\le \tau}(x)$ or $\frac{x}{1+|x|^2}$.
They affect only the values of $\gamma$ in the L\'{e}vy-Khintchine 
representations
$$
\log\psi(t)=i <\gamma_1, t>+\int_{|x|\le \tau}(e^{i <t, x>}-1-
i<t, x>)+\int_{|x|>\tau}(e^{i <t, x>}-1)dM(x)
$$
or
$$
\log\psi(t)=i  <\gamma_2,  t>+\int\big(e^{i <t, x>}-1-\frac{i 
<t, x>}{1+|x|^2}\big)dM(x)
$$
with
$$
\gamma_1=\gamma_2+\int\big(x {\bf 1}_{|x|\le \tau}-\int \frac {x}{1+
|x|^2}\big)dM(x)
$$
If we want to relate the truncated mean of $X+Y$ to the sum of the truncated 
means of $X$ and $Y$ it is easier to handle $f(X+Y)-f(X)-f(Y)$
with $f(x)=\frac{x}{1+|x|^2}$ than with $f(x)=x {\bf 1}_{|x|\le \tau}$. 
If we use the truncated mean with some $f(x)$ and center by subtracting the 
truncated mean then we end up in the limit  with the representation
$$
\log\psi(t)=\int[e^{i\ <t, x>}-1- i <t, f(x)>]dM(x)
$$
with $\gamma=0$. We can do one truncation and still go to the other 
representation by defining $\gamma$ suitably (cf. \cite{GK}).
\end{remark}
\begin{remark}\label{rem2.6}
An obvious corollary of Theorems \ref{thm2.3} and \ref{thm2.4} when
$\al_*>1$ is a weak law of large numbers saying that for all $\te\in\Te$,
\[
\lim_{N\to\infty}\frac 1NS_N(\te,1)=EZ_\te
\]
where convergence is considered in probabability. This is not new in the 
$\ell$-dependence case since then the summands in $S_N(\te,1)$ form a 
stationary sequence but in the arithmetic progression case of Theorem 
\ref{thm2.4} this assertion does not seem to follow directly from previous
results.
\end{remark}

\section{Tails of products of independent random variables}\label{sec3}
\setcounter{equation}{0}

Clearly, in order to establish Theorem \ref{thm2.1} for products, i.e.
$F(x_1,x_2,...,x_\ell)=x_1x_2\cdots x_\ell$ it suffices to prove it for
$\ell=2$ and then to proceed by induction. Thus, we prove first
\begin{proposition}\label{prop3.1}
Let $X_1$ and $X_2$ be independent random variables such that $X_1$ satisfies
(\ref{2.1}) with $0<\al_1<2,\, k_1\geq 0$ and $c_1^\pm\geq 0,\,
c_1^++c_1^->0$.

(i) Suppose that for some $\al_2>\al_1$,
\begin{equation}\label{3.1}
\lim\sup_{z\to\infty}z^{\al_2}P\{ |X_2|>z\}=\rho<\infty.
\end{equation}
Then
\begin{eqnarray}\label{3.2}
&\lim_{z\to\infty}z^{\al_1}(\ln z)^{-k_1}P\{\pm X_1X_2>z\}\\
&=\al_1\int_0^\infty x^{\al_1-1}(c_1^+P\{\pm X_2>x\}+c^-_1
P\{\mp X_2>x\})dx.
\nonumber\end{eqnarray}

(ii) Suppose that $X_2$ satisfies (\ref{2.1}) with $\al_2=\al_1\in(0,2)$ and
 some $k_2\in\bbZ_+\cup\{ 0\}$ and $c_2^\pm\geq 0,\,c_2^++c_2^->0$. Then
\begin{eqnarray}\label{3.3}
&\lim_{z\to\infty}z^{\al_1}(\ln z)^{-(k_1+k_2+1)}P\{\pm X_1X_2>z\}\\
&=\al_1c^\pm\frac {\Gam(k_1+1)\Gam(k_2+1)}{\Gam(k_1+k_2+2)}
\nonumber\end{eqnarray}
where $c^+=c^+_1c^+_2+c^-_1c^-_2$, $c^-=c^+_1c^-_2+c^-_1c^+_2$ and $\Gam$
denotes the Gamma function.
\end{proposition}
\begin{proof}
The assertion (i) follows actually from the theorem of Breiman
as presented by means of Proposition 7.5 in \cite{Re}. On the other hand, the
assertion (ii) seems to be specific for the class of 
random variables satisfying (\ref{2.1}) and for readers' convenience we will
 give the complete proof of the whole result here. In the proof we will employ 
 several times integration by parts for Stiltjes integrals which will be
 legitimate since integrands in our circumstances will be differentiable
 (see, for instance, \cite{Bi}, Theorem 18.4 and remarks there or \cite{Sh},
 Theorem 11 and Corollary 1 in \S 6, Ch.II). For any $z>1$ we write
\begin{equation}\label{3.4}
Q(z)=P\{ X_1X_2>z\}=Q_1^+(z)+Q_1^-(z)+R_1(z)
\end{equation}
where for some small $\del>0$,
\[
Q_1^+(z)=E{\bf 1}_{0<X_2\leq z/(\ln z)^\del}P\{X_1>\frac z{X_2}|X_2\},
\]
\[
Q_1^-(z)=E{\bf 1}_{0>X_2\geq -z/(\ln z)^\del}P\{-X_1>\frac z{-X_2}|X_2\}
\]
and by (\ref{3.1}) or by (\ref{2.1}) depending on the case,
\begin{equation}\label{3.5}
|R_1(z)|\leq P\{ |X_2|>\frac z{(\ln z)^\del}\}\leq C_1z^{-\al_2}
(\ln z)^{\al_2\del}\big(\ln\frac z{(\ln z)^\del}\big)^{k_2}
\end{equation}
for some $C_1>0$ independent of $z$. Next, set
\begin{eqnarray*}
&Q^+_2(z)=c^+_1E{\bf 1}_{0<X_2<z/(\ln z)^\del}\big(\frac {X_2}z\big)^{\al_1}
\big(\ln\frac z{X_2}\big)^{k_1}=c_1^+I_1^+(z)\\
&\mbox{where}\,\, I_1^+(z)=\int_0^{z/(\ln z)^\del}(\frac xz)^{\al_1}(\ln
\frac zx)^{k_1}dP\{ X_2\leq x\}.
\end{eqnarray*}
Then taking into account that $\frac z{X_2}\geq(\ln z)^\del$ when $0<X_2
<z/(\ln z)^\del$ we obtain that
\begin{equation}\label{3.6}
|Q_1^+(z)-Q_2^+(z)|\leq I_1^+(z)R_2(z)
\end{equation}
where by (\ref{2.1}),
\begin{equation}\label{3.7}
R_2(z)=\sup_{u\geq(\ln z)^\del}|u^{\al_1}(\ln u)^{-k_1}P\{ X_1>u\}-c^+_1|
\to 0\,\,\mbox{as}\,\, z\to\infty.
\end{equation}

Now integrating by parts we obtain
\begin{eqnarray}\label{3.8}
&I_1^+(z)=z^{-\al_1}x^{\al_1}(\ln\frac zx)^{k_1}P\{ X_2\leq
 x\}|_0^{z/(\ln z)^\del}\\
&-z^{-\al_1}\int_0^{z/(\ln z)^\del}x^{\al_1-1}\big(\al_1(\ln\frac zx)^{k_1}-
k_1(\ln\frac zx)^{k_1-1}\big)P\{ X_2\leq x\}dx=Q_3^+(z)\nonumber
\end{eqnarray}
where
\[
Q_3^+(z)=z^{-\al_1}\int_0^{z/(\ln z)^\del}x^{\al_1-1}\big(\al_1
(\ln\frac zx)^{k_1}-
k_1(\ln\frac zx)^{k_1-1}\big)P\{\frac z{(\ln z)^\del}>X_2>x\}dx.
\]
Set
\[
Q_4^+(z)=z^{-\al_1}\int_0^{z/(\ln z)^\del}x^{\al_1-1}\big(\al_1
(\ln\frac zx)^{k_1}-
k_1(\ln\frac zx)^{k_1-1}\big)P\{ X_2>x\}dx.
\]
Then
\begin{equation}\label{3.9}
|Q_3^+(z)-Q_4^+(z)|\leq R_3(z)P\{ |X_2|>z/(\ln z)^\del\}
\end{equation}
where changing variables $y=\ln\frac zx$ we have
\[
R_3(z)=|\int_{\del\ln\ln z}^\infty e^{-\al_1y}(\al_1y^{k_1}-k_1y^{k_1-1})dy|.
\]

Assuming $\al_2>\al_1$ we have from (\ref{3.1}) that $x^{\al_2}P\{ X_2>x\}$ is
bounded and then it follows easily that
\begin{equation}\label{3.10}
\lim_{z\to\infty}z^{\al_1}(\ln z)^{-k_1}Q^+_4(z)=\al_1\int_0^\infty x^{\al_1-1}
P\{ X_2>x\}dx
\end{equation}
and the integral in (\ref{3.10}) converges. On the other hand, when 
$\al_2>\al_1$, it follows from (\ref{3.1}) and (\ref{3.5})--(\ref{3.9}) that
\begin{equation}\label{3.11}
\lim_{z\to\infty}z^{\al_1}(\ln z)^{-k_1}|Q^+_1(z)-c_1^+Q_4^+(z)|=0.
\end{equation}
This together with (\ref{3.10}) and similar estimates for $Q^-_1(z)$ yields
(\ref{3.2}) proving (i).

Next we complete the proof of (ii) considering the case $\al_2=\al_1$. By 
(\ref{2.1}),
\begin{equation}\label{3.12}
R_4(u)=\sup_{w\geq u}|c^+_2-w^{\al_2}(\ln w)^{-k_2}P\{ X_2>w\}|\to 0\,\,
\mbox{as}\,\, u\to\infty.
\end{equation}
For each $\ve>0$ choose $u_\ve\geq 1$ so that
\begin{equation}\label{3.13}
R_4(u_\ve)\leq\ve\,\,\mbox{and}\,\, u_\ve\to\infty\,\,\mbox{as}\,\,\ve\to 0.
\end{equation}
If $z$ is large enough so that $z/(\ln z)^\del>u_\ve$ then setting $y=\ln x$ 
we write
\begin{eqnarray*}
&Q_5^{(\ve)}(z)=c_2^+I_2^{(\ve)}(z)\,\,\mbox{where}\\
&I^{(\ve)}_2(z)=z^{-\al_1}\int_{u_\ve}^{z/(\ln z)^\del}x^{-1}
(\ln x)^{k_2}\big(\al_1(\ln\frac zx)^{k_1}-k_1(\ln\frac zx)^{k_1-1}\big)dx\\
&=c^+_2z^{-\al_1}\int_{u_\ve}^{\ln z-\del\ln
\ln z}y^{k_2}\big(\al_1(\ln z-y)^{k_1}
-k_1(\ln z-y)^{k_1-1}\big)dy
\end{eqnarray*}
and
\[
Q_6^{(\ve)}(z)=z^{-\al_1}\int_0^{u_\ve}x^{\al_1-1}
\big(\al_1(\ln\frac zx)^{k_1}-k_1(\ln\frac zx)^{k_1-1}\big)P\{ X_2>x\}dx.
\]
Then
\begin{equation}\label{3.14}
|Q_4^+(z)-Q_5^{(\ve)}(z)-Q_6^{(\ve)}(z)|\leq\ve I_2^{(\ve)}..
\end{equation}
Changing variables $u=\frac y{\ln z}$ and integrating by parts repeatedly
we obtain
\begin{eqnarray}\label{3.15}
&\lim_{z\to\infty}z^{\al_1}(\ln z)^{-(k_1+k_2+1)}I_2^{(\ve)}(z)\\
&=\al_1\int_0^1u^{k_2}(1-u)^{k_1}du=\al_1\frac {\Gam(k_1+1)\Gam(k_2+1)}
{\Gam(k_1+k_2+2)}\nonumber
\end{eqnarray}
since the last integral above is the well known $\beta$-function $B(k_1+1,
k_2+1)$. On the other hand,
\begin{equation}\label{3.16}
\lim_{z\to\infty}z^{\al_1}(\ln z)^{-(k_1+k_2+1)}Q_6^{(\ve)}(z)=0.
\end{equation}

Now observe that by (\ref{3.5}) choosing $\del<\frac {k_1}{\al_1}$ we have
\begin{equation}\label{3.17}
\lim_{z\to\infty}z^{\al_1}(\ln z)^{-(k_1+k_2+1)}|R_1(z)|=0.
\end{equation}
Similarly, by (\ref{3.5}) and (\ref{3.9}),
\begin{equation}\label{3.18}
\lim_{z\to\infty}z^{\al_1}(\ln z)^{-(k_1+k_2+1)}|Q^+_3(z)-Q^+_4(z)|=0.
\end{equation}

It follows from (\ref{3.8}), (\ref{3.9}), (\ref{3.12})--(\ref{3.16}) 
and (\ref{3.18}) that the (finite) limit
\[
\lim_{z\to\infty}z^{\al_1}(\ln z)^{-(k_1+k_2+1)}Q_2^+(z)
\]
exists. Since  $\lim_{z\to\infty} R_2(z)=0$ by (\ref{3.7}) we see from 
(\ref{3.6}) that
\begin{equation}\label{3.19}
\lim_{z\to\infty}z^{\al_1}(\ln z)^{-(k_1+k_2+1)}|Q^+_1(z)-Q^+_2(z)|=0.
\end{equation}
Finally, (\ref{3.8}) and (\ref{3.14})--(\ref{3.19}) together with similar
 estimates for $Q^-_1(z)$ yields (\ref{3.3}) with $"+"$ while (\ref{3.3})
  with $"-"$ follows in the same way.
\end{proof}

\begin{remark}\label{rem3.2}
Suppose that $X_1$ and $X_2$ are positive random variables having the tail
 behavior $P\{ X_i>x\}\sim c_ix^{-\al_i}(\ln x)^{k_i}$, $i=1,2$ as $x\to\infty$
 with $\al_i,c_i>0$ and $k_i\geq 0$, $i=1,2$. Then $Y_i=\ln X_i$, $i=1,2$ have
 the tail behavior $P\{ Y_i>y\}\sim c_ie^{-\al_i}y^{k_i}$, $i=1,2$ as 
 $y\to\infty$. The latter tail behavior (with a proper normalization) is the 
 same as for the $\Gam$-distribution with parameters $k_i+1$ and $\al_i^{-1}$.
 Thus, we can first study the tail behavior of the sum $Y_1+Y_2$ of two
 independent random variables $Y_1$ and $Y_2$ having the $\Gam$-distribution
 with parameters $k_1+1,\al_1^{-1}$ and $k_2+1,\al^{-1}$, respectively, and
 then consider $\exp(Y_1+Y_2)$. This is strightforward when $\al_1=\al_2=\al$
 since then $Y_1+Y_2$ has the $\Gam$-distribution with parameters $k_1+k_2+2,\,
 \al^{-1}$. If $\al_1\ne\al_2$ then one has to obtain the tail behavior of 
 a convolution of two $\Gam$-distributions which involves computation of some
 integrals. Of course, we can take the logarithm only if $X_1$ and $X_2$ are
 positive but negative values can be treated similarly considering positive
 tails of $-X_1$ and $-X_2$. Still, one has to take special care of the cases
 when some of $c_i$'s are zero and when $X_i$'s may take on zero values.
 \end{remark}

Proceeding by induction in the number of variables we derive from Proposition
\ref{prop3.1} the following result.
\begin{corollary}\label{cor3.3}
Theorem \ref{thm2.1} holds true for any monomial 
\[
F(X_1,...,X_\ell)=X_1^{\sig_1}X_2^{\sig_2}\cdots X_\ell^{\sig_\ell},\,\, 
\sig_j\in\bbZ_+
\]
where $X_1,...,X_\ell$ are independent random variables satisfying (\ref{2.1}).
\end{corollary}
\begin{proof}
Observe first that the assertions (i) and (ii) of Theorem \ref{thm2.1} 
coincide here since now $F$ consists of only one monomial, and so the
descriptions of the tail behavior (\ref{2.3}) and (\ref{2.4}) are 
equivalent in view of Theorem 3.6 in \cite{Re} and Lemma \ref{l5.1} of
Appendix.

Next, set $Y_i=X_i^{\sig_i}$, $i=1,...,\ell$. Then (\ref{2.1}) implies that for 
$i=1,...,\ell$,
\begin{equation}\label{3.20}
\lim_{z\to\infty}z^{\al_i/\sig_i}(\ln z)^{-k_i}P\{ -Y_i>z\}=0\,\,\,\mbox{if}
\,\, \sig_i\,\,\,\mbox{is even and}
\end{equation}
\begin{equation}\label{3.21}
\lim_{z\to\infty}z^{\al_i/\sig_i}(\ln z)^{-k_i}P\{\pm Y_i>z\}=
\lim_{z\to\infty}z^{\al_i/\sig_i}(\ln z)^{-k_i}P\{ \pm X_i>z^{1/\sig_i}\}=
\sig_i^{-k_i}c_i^\pm
\end{equation}
provided $\sig_i$ is odd and, furthermore, if $\sig_i>0$ is even then
\begin{eqnarray}\label{3.22}
&\lim_{z\to\infty}z^{\al_i/\sig_i}(\ln z)^{-k_i}P\{ \pm Y_i>z\}
=\lim_{z\to\infty}\\
&z^{\al_i/\sig_i}(\ln z)^{-k_i}(P\{ X_i>z^{1/\sig_i}\}+
P\{ X_i<-z^{1/\sig_i}\})=\sig_i^{-k_i}(c_i^++c_i^-).\nonumber
\end{eqnarray}
Next, we proceed by induction in $\ell$. For $\ell=1$ the result follows
from (\ref{3.20})--(\ref{3.22}). Suppose that it still holds true for
$\ell=1,2,...,n-1$. In order to obtain it for $\ell=n$ we set 
$Z=Y_1\cdots Y_{n-1}$. Then by the induction hypothesis 
\[
\lim_{z\to\infty}z^\al(\ln z)^{-k}P\{\pm Z>z\}=c^\pm
\]
for some $\al,\, k$ and $c^\pm$ described in Theorem \ref{thm2.1}. Since $Y_n$
 satisfies (\ref{3.20})--(\ref{3.22}) we derive the result for $ZY_n$ from 
 Proposition \ref{prop3.1} completing both the induction step and the proof
 of this corollary.
\end{proof}

Next, we derive (\ref{2.6}) of Example \ref{ex2.2}. We write
\begin{equation}\label{3.23}
Q_0(z)=P\{ X_1X_2>z\}=2Q_1(z)+R_1(z)
\end{equation}
where $Q_1(z)=Q_1^+(z)=Q_1^-(z)$ and $R_1(z)$ are the same as in (\ref{3.4})
and now
\begin{equation}\label{3.24}
|R_1(z)|\leq P\{ |X_2|>\frac z{(\ln z)^\del}\}\leq Cz^{-\al}(\ln z)^{\al
\del}\big(\ln\frac z{(\ln z)^\del}\big)^{-1}
\end{equation}
for some $C>0$.
Similarly to (\ref{3.6}) we have also
\begin{equation}\label{3.25}
|Q_1^+(z)-Q_2^+(z)|\leq\frac 1{c^+_1}Q_2^+(z)R_2(z)
\end{equation}
where 
\begin{equation}\label{3.26}
Q_2(z)=c_1E{\bf 1}_{0<X_2<\frac z{(\ln z)^\del}}(\frac {X_2}z)^\al=c_1z^{-\al}
\int_0^{z/(\ln z)^\del}x^\al dP\{ X_2\leq x\}
\end{equation}
and
\[
R_2(z)=\sup_{u\geq(\ln z)^\del}|u^{\al}P\{ X_1>u\}-c_1|.
\]
Next, we write
\begin{equation}\label{3.27}
Q_2(z)=Q_3(z)-R_3(z)
\end{equation}
where
\[
Q_3(z)=c_1\al z^{-\al}\int_0^{z/(\ln z)^\del}x^{\al-1}dP\{ X_2>x\}dx-R_3(z)
\]
and by (\ref{3.23}),
\begin{equation}\label{3.28}
R_3(z)=\frac {c_1}{(\ln z)^{\del\al}}P\{ X_2>\frac z{(\ln z)^\del}\}\leq
Cc_1z^{-\al}\big(\ln\frac z{(\ln z)^\del}\big)^{-1}.
\end{equation}

Now, we write
\begin{equation}\label{3.29}
Q_3(z)=Q_4^{(\ve)}(z)+Q_5^{(\ve)}(z)+R_4^{(\ve)}(z)
\end{equation}
where
\begin{equation}\label{3.30}
Q_4^{(\ve)}(z)=c_1\al z^{-\al}\int_0^{u_\ve}x^{\al-1}P\{ X_2>x\} dx
\leq c_1z^{-\al}u_\ve^\al,
\end{equation}
\begin{equation}\label{3.31}
Q_5^{(\ve)}(z)=c_1c_2\al z^{-\al}\int_{u_\ve}^{z/(\ln z)^\del}x^{-1}
(\ln x)^{-1}dx=c_1c_2\al z^{-\al}\ln\big(\frac {\ln-\del\ln\ln z}
{\ln u_\ve}\big)
\end{equation}
and
\begin{equation}\label{3.32}
|R^{(\ve)}_4(z)|\leq\frac \ve{c_2\al} Q_5^{(\ve)}(z)=\ve c_1z^{-\al}
\ln\big(\frac {\ln-\del\ln\ln z}{\ln u_\ve}\big)
\end{equation}
provided $u_\ve\to\infty$ as $\ve\to 0$ is chosen so that 
\[
\sup_{w\geq u_\ve}|c_2-w^\al(\ln w)P\{ X_2>w\}|\leq\ve.
\]

By (\ref{3.29})--(\ref{3.32}),
\begin{equation}\label{3.33}
\lim_{z\to\infty}z^{\al}(\ln\ln z)^{-1}Q_3(z)=\lim_{\ve\to 0}\lim_{z\to\infty}
Q_5^{(\ve)}(z)=c_1c_2\al
\end{equation}
which together with (\ref{3.23})--(\ref{3.28}) yields (\ref{2.6}). \qed

\section{Tails of polynomials}\label{sec4}\setcounter{equation}{0}

In order to prove Theorem \ref{thm2.1} we will view the collection of
monomials which compose the polynomial $F$ as a vector and fit them 
into the following setup. For $i=1,2,...,M$ let $Z_i=(Z_{i1},Z_{i2},...,
Z_{im_i})$ be $m_i$-dimensional random vectors and $\mathcal W_i$ be the
spaces of bounded continuous functions $W$ on $R^{m_i}$ satisfying 
(\ref{2.1+}). Suppose that for each $i$ and any $W\in\cW_i$,
\begin{equation}\label{4.1}
\lim_{\rho\to\infty}\rho^\al(\ln\rho)^{-k}E[W(\frac {Z_i}\rho)]=
\int_0^\infty\int_{S^{m_i-1}}W(sz)\frac {\nu_i(ds)dz}{|z|^{1+\al}}
\end{equation}
where $\al>0$ and $\nu_i$ are measures on $S^{m_i-1}$, $i=1,...,M$ with 
$\nu_i(S^{m_i-1})>0$. Let also $Z_0=(Z_{01},...,Z_{0,m_0})$ be a 
$m_0$-dimensional random vector such that
\begin{equation}\label{4.2}
\lim_{\rho\to\infty}\rho^\al(\ln\rho)^{-k}P\{ |Y_0|>\rho\}=0.
\end{equation}

\begin{lemma}\label{lem4.1} With the above notations consider the
$m=\sum_{i=0}^Mm_i$-dimensional random vector $Z=(Z_0,Z_1,...,Z_M)$ and
assume that
\begin{equation}\label{4.3}
\lim_{\rho\to\infty}\rho^\al(\ln\rho)^{-k}P\{ |Z_{il_1}|>\rho,\,
|Z_{jl_2}|>\rho\}=0
\end{equation}
for any $i,j=1,...,M,\, i\ne j$ and $1\leq l_1\leq m_i,\, 1\le l_2\leq m_j$.
Then for each bounded continuous $W$ on $R^m$ satisfying (\ref{2.1+})
the limit,
\begin{equation}\label{4.4}
\lim_{\rho\to\infty}\rho^\al(\ln\rho)^{-k}E[W(\frac Z\rho)]=\int_0^\infty
\int_{S^{m-1}}W(sx)\frac {\nu(ds)dx}{|x|^{1+\al}}
\end{equation}
exists with a measure $\nu$ supported on $S^{m-1}\cap (\cup_{i=1}^M\Gam_i)$
where $\Gam_i=(0,...,0,R^{m_i},0,...,0),\, i=1,...,M$. Moreover, the 
projection to $R^{m_i}$ of the restriction $\nu |_{\Gam_i}$ coinsides with 
$\nu_i,\, i=1,...,M$.
\end{lemma}
\begin{proof} 
In view of (\ref{4.2}) and (\ref{4.3}) any weak limit as $\rho\to\infty$ of 
the distributions
\[
\rho^\al(\ln\rho)^{-k}P\{\frac Z\rho\in\cdot\}
\]
has support on $\cup_{i=1}^M\Gam_i$, and by (\ref{4.1}) the limiting measure
$\nu$ exists and the projection of $\nu |_{\Gam_i}$ to $R^{m_i}$ coincides
with $\nu_i$.
\end{proof}

We will need also the following result where we denote by $\cW$ and $\cW_m$
the spaces of bounded continuous functions on $R$ and $R^m$, respectively,
satisfying (\ref{2.1+}).

\begin{lemma}\label{lem4.2}
Let  a scalar random variable  $V$ be such that for some $\alpha, k$ with 
$0<\alpha<2$, $k\ge 0$ and all $W\in\cW$,

 $$
\lim_{\rho\to\infty} \rho^\alpha(\ln \rho)^{-k}E[W\big(\frac{V}{\rho}\big)]=
\int_{-\infty}^\infty (c_-{\bf 1}_{x<0}+c_+{\bf 1}_{x>0})W(x)\frac{dx}{|x|^{1+
\alpha}}.
$$
Next, let $Y$ be  a random vector in $R^m$ independent of $V$ and satisfying
$E[|Y|^\alpha]<\infty$.  Then the tail behavior of the  vector $Z=VY$ is
given by the limit
\begin{equation}\label{4.5}
\lim_{\rho\to\infty} \rho^\al(\ln \rho)^{-k} E[ W(\frac{Z}{\rho})]=
\int_0^\infty\int_{S^{m-1}} W(sx)\frac{\nu(ds)\,dx}{|x|^{1+\alpha}}
\end{equation}
which holds true for any $W\in\cW_m$. The measure $\nu(ds)$ on $S^{m-1}$ is 
computed for subsets $A\subset S^{m-1}$ from the identity
$$
E[\int_0^\infty (c_+W(xY)+c_-W(-xY)] \frac{dx}{x^{1+\alpha}})dx=\int_{S^{m-1}}
\int_0^\infty W(sr)\frac{\nu(ds)dr}{r^{1+\alpha}}
$$
as
$$
\nu(A)=c_+\int_{\widehat A} |{\bf y}|^\alpha d\lambda+c_-\int_{-\widehat A} 
|{\bf y}|^\alpha d\lambda
$$
where ${\widehat A}$ is the cone $\cup_{\sigma>0}\sigma A$, ${-\widehat A} =
\cup_{\sigma<0} \,\sigma A$ and $\lambda$ is the distribution of $Y$ on $R^m$.
\end{lemma}
\begin{proof} Given $W\in \mathcal W$, for each fixed ${\bf y}\in R^m$ as a
function of $x$, $W(x{\bf y})\in \mathcal W$ on $R$.  Therefore
$$
\lim_{\rho\to\infty}\rho^\alpha (\ln \rho)^{-k} E[W(\frac{V{\bf y}}{\rho})]=
\int_0^\infty (c_+W(x{\bf y})+c_-W(-x{\bf y})) \frac{dx}{x^{1+\alpha}}
$$
Moreover, the convergence is easily seen to be uniform over  bounded sets of 
${\bf y}$. Hence if  $\rho\ge 3$,
$$
\sup_{\rho\ge 3}\sup_{\|{\bf y}\|\le 1}\rho^\alpha (\ln \rho)^{-k} 
E[W(\frac{V{\bf y}}{\rho})]\le C.
$$ 
If $1\le |{\bf y}|\le \frac{\rho}{3}$ then setting  
$\rho'=\frac{\rho}{|{\bf y}|}
\ge 3$, writing ${\bf y}=|{\bf y}|{\bf y'}$ with $|{\bf y}'|=1$ and observing
 that $\rho'\le \rho$ we obtain
\begin{eqnarray*}
&\rho^\alpha (\ln \rho)^{-k} E[W(\frac{V{\bf y}}{\rho})] =\rho^\alpha 
(\ln \rho)^{-k} E[W(\frac{V{\bf y}'}{\rho |{\bf y}|^{-1}})]
=\rho^\alpha (\ln \rho)^{-k} E[W(\frac{V{\bf y}'}{\rho'})]\\
&\le \frac{\rho^\alpha (\ln \rho)^{-k} }{(\rho')^\alpha (\ln \rho')^{-k} }
(\rho')^\alpha (\ln \rho')^{-k}E[W(\frac{V{\bf y}'}{\rho'})]
\le C\frac{\rho^\alpha (\ln \rho)^{-k} }{(\rho')^\alpha (\ln \rho')^{-k} }
\le C |{\bf y}|^\alpha. 
\end{eqnarray*}
If $|{\bf y}|\ge \frac{\rho}{3}$ and $\ln \rho\ge 1$ then 
$$
\rho^\alpha (\ln \rho)^{-k} E[W(\frac{V{\bf y}}{\rho})]\le \|W\|_\infty 
\rho^\alpha \le  C|{\bf y}|^\alpha.
$$
We can now apply the dominated convergence theorem to conclude that
$$
\lim_{\rho\to\infty}\rho^\alpha (\ln \rho)^{-k} E[W(\frac{VY}{\rho})]=
E\int_0^\infty (c_+W(xY)+c_-W(-xY)) \frac{dx}{x^{1+\alpha}}.
$$
Replacing $W(\cdot)$ by $W(\sigma\, \cdot)$ and changing variables we see 
that
$$
E\int_0^\infty (c_+W(xY)+c_-W(-xY)) \frac{dx}{x^{1+\alpha}}
$$
 is homogeneous of order $\alpha$ under dilation and therefore can be 
 expressed as 
 $$
 \int_{S^{m-1}} \int_0^\infty  W(rs)\frac{\nu(ds)dr}{r^{1+\alpha}}
$$
where for a Borel $A\subset S^{m-1}$,
$$
\nu(A)=c_+\int_{\widehat A} |{\bf y}|^\alpha d\lambda+c_-\int_{-\widehat A} 
|{\bf y}|^\alpha d\lambda
$$
with $\lambda$ being the distribution of $Y$.  If $Y$ is a scalar then the
 L\'{e}vy measure on $R$ is given for each Borel $A\subset R$ by
$$
\nu(A)=\int_A(c^*_- {\bf 1}_{x<0}+c^*{\bf 1}_{x>0})\frac{dx}{|x|^{1+\alpha}}
$$
where
 $$
 c^*_-=c_+\int_{-\infty}^0 |y|^\alpha d\lambda+c_-\int_{0}^\infty |y|^\alpha  
 d\lambda
 $$
 and
  $$
  c^*_+=c_-\int_{-\infty}^0 |y|^\alpha d\lambda+c_+\int_{0}^\infty 
  |y|^\alpha d\lambda.
  $$
 \end{proof}

In order to apply Lemma \ref{4.1} to the collection of monomials composing
the polynomial $F$  we will need the following result which will
ensure the compliance with the condition (\ref{4.3}).

\begin{lemma}\label{lem4.3}
Let $Y=V_1V_2$ and $Z=V_2V_3$ where $V_1,\, V_2$ and $V_3$ are independent
random variables such that 
\begin{equation}\label{4.6}
\lim_{z\to\infty}z^\al(\ln z)^{-u_i}P\{\pm V_i>z\}=v^{\pm}_i
\end{equation}
where $\al>0,\, u_i\geq 0,\, v_i^++v_i^->0,\, i=1,2,3$. Then
\begin{equation}\label{4.7}
\limsup_{z\to\infty}z^\al(\ln z)^{-u_2}P\{ |Y|>z,\, |Z|>z\}<\infty.
\end{equation}
\end{lemma}
\begin{proof} Observe that
\begin{equation}\label{4.8}
P\{ |Y|>z,\, |Z|>z\}\leq P\{ |YZ|^{1/2}>z\}=P\{ |V_1V_2|^{1/2}|V_2|>z\}.
\end{equation}
By (\ref{4.6}) and Proposition \ref{prop3.1},
\[
\limsup_{z\to\infty}z^{2\al}(\ln z)^{-(u_1+u_3+1)}P\{ |V_1V_2|^{1/2}>z\}
<\infty,
\]
and so, again, by (\ref{4.6}) and Proposition \ref{prop3.1},
\[
\limsup_{z\to\infty} z^\al(\ln z)^{-u_2}P\{ |V_1V_3|^{1/2}|V_2|>z\}<\infty
\]
which together with (\ref{4.8}) gives (\ref{4.7}).
\end{proof}

\begin{corollary}\label{cor4.3+}
Let $\theta_1,\theta_2$ be two multi indices  with corresponding monomials 
$g_{\theta_1}=\Pi_{j=1}^\ell x_j^{ \sigma_j}$ and $g_{\theta_2}=\Pi_{j=1}^\ell
 x_j^{\tau_j}$. If $\alpha(\theta_1)=\alpha(\theta_2)=\alpha^\ast$ and 
 $k(\theta_1)=k(\theta_2)=k^\ast$ but $J(\theta_1)\not= J(\theta_2)$, then
$$
\lim_{x\to\infty}  x^{\alpha^\ast}(\log x)^{-k^\ast} P[|g_{\theta_1}(X_1,
\ldots,X_\ell) |\ge  x, \,|g_{\theta_2}(X_1,\ldots,X_\ell )|\ge  x]=0
$$
\end{corollary}
\begin{proof} In Lemma \ref{lem4.3} set $Y=g_{\te_1}(X_1,...,X_\ell)$,
$Z=g_{\te_1}(X_1,...,X_\ell)$ and $V_2=\prod_{i=1}^mX_{j_i}$ where
$j_1,...,j_m\in J(\te_1)\cap J(\te_2)$ while if the latter intersection
is empty we take $V=1$. Now, the result follows immediately from Lemma 
\ref{lem4.3}.
\end{proof} 

Now we are able to complete the proof of Theorem \ref{thm2.1}. Order
 arbitrarily the sets $J(\te),\,\te\in\Te_*$ and denote these different
  sets by $J_1,J_2,...,J_M$. Next, we define vectors $Z_i,\, i=1,...,M$ so that
  $Z_i$ consists of different monomials $g_\te,\,\te\in\Te_*$ having the form
  \begin{equation}\label{4.9}
  g_\te=(\prod_{j\in J_i}x_j)^{\sig_*}\prod_{l\not\in J_i,\sig_l\in\te}
  x^{\sig_l}_l
  \end{equation}
  observing that $\sig_l<\sig_*$ here for all $l\not\in J_i$. Define
  also the vector $Z_0$ which consists of monomials $g_\te,\,\te\in\Te
  \setminus\Te_*$ taken with an arbitrary order. In order to obtain
  (\ref{4.1}) for $Z_i,\, i\geq 1$ consider random variables
   $V_i=(\prod_{j\in J_i}X_j)^{\sig_*}$ and form random vectors $Y_i$ which
   consist of random monomials $\prod_{l\not\in J_i,\sig_l\in\te}X_l^{\sig_l}$,
   so that $Z_i=V_iY_i$. Next, the conditions of Lemma \ref{lem4.2} are
   verified relying on Corollary \ref{cor3.3} which enables us to apply 
   Lemma \ref{lem4.2} in order to
   obtain (\ref{4.1}) with $\al=\al_*$ and $k=k_*$. Now, for such $\al$ and
    $k$ the condition (\ref{4.2}) follows from Corollary \ref{cor3.3} and the 
   condition (\ref{4.3}) follows from Lemma \ref{lem4.3} using 
   Corollary \ref{cor3.3}. Hence, applying Lemma \ref{lem4.1} we derive
   the assertion (i) of Theorem \ref{thm2.1}. Since taking the sum of
   components of vectors is a particular case of a linear map the assertion 
   (ii) of Theorem \ref{thm2.1} follows from (i), Lemma \ref{l5.1} and the
    following observation.
   
   \begin{lemma}\label{lem4.4} (see Proposition 7.3 of \cite{Re}) Suppose 
   that a random $m$-dimensional vector
   $Z$ satisfies (\ref{4.4}) for some measure $\nu$ on $S^{m-1}$ with
   $\nu(S^{m-1})>0$ and any bounded continuous function $W$ satisfying
   (\ref{2.1+}). Let $T:R^m\to R^d$, $d\leq m$ be a linear map. Then
   $Z'=TZ$ will again satisfy (\ref{4.4}) with $Z',\nu'$ and $d$ in place
   of $Z,\nu$ and $m$ where $\nu'$ is defined for any Borel set $\Gam\subset
   S^{d-1}$ by
   \[
   \nu'(\Gam)=\int_{S^{m-1}}{\bf 1}_\Gam(\frac {Ts}{|Ts|})|Ts|^\al\nu(ds).
   \]
   \end{lemma}
   \begin{proof} The result follows considering bounded continuous functions
   $\tilde W$ on $R^d$ which satisfy (\ref{2.1+}) while observing that $W(x)
   =\tilde W(Tx)$ is a bounded continuous function on $R^m$ satisfying 
   (\ref{2.1+}) to which we can apply (\ref{4.4}).
   \end{proof}

 \section{Limit theorem in the $\ell$-dependence case}\label{sec5}
  \setcounter{equation}{0}

In this section we will prove Theorem \ref{thm2.3}. We are dealing here 
with the sums
\begin{equation}\label{5.1}
S_N(\te,t)=\sum_{1\leq n\leq Nt}g_\te(X_n,X_{n+1},...,X_{n+\ell-1}),\,
\te\in\Te
\end{equation}
containing the summands which form a stationary $\ell$-dependent sequence. 
By this reason our proof will rely on the following result which appears
in \cite{TK} as Corollary 1.4.

\begin{proposition}\label{prop5.1} Let $\{ Z_n\}$ be a stationary
$\ell$-dependent sequence with values in $R^d$ which satisfies (\ref{4.1})
and such that for any $j=2,3,...,\ell$ and $\del>0$,
\begin{equation}\label{5.2}
\lim_{N\to\infty}NP\{ |Z_1|>\del b_N\,\,\mbox{and}\,\, |Z_j|>\del b_N\}=0
\end{equation}
where $b_N=N^{\frac 1\al}(\frac {\ln N}\al)^{\frac k\al}$. Set $a_N=
E(\frac {b_N^2Z_1}{b_N^2+|Z_1|^2})$. Then the process
\begin{equation}\label{5.3}
\Up_N(t)=\frac 1{b_N}(\sum_{1\leq n\leq Nt}Z_n-Nta_N)
\end{equation}
weakly converges in the $J_1$ topology on $D([0,T],\, R^d)$ to an $\al$-stable
L\' evy process with the same L\' evy measure as $Z_1$.
\end{proposition}

 Actually, the identification of the limiting L\' evy measure is not stated 
 explicitly in Corollary 1.4 of \cite{TK} but this follows from the proof there.
In Appendix we will exhibit a more general result for the $\ell$-dependent 
stationary sequences which will yield either convergence in the $J_1$ topology
or only of finite dimensional distributions depending on whether the condition (\ref{5.2}) 
is assumed or not and we will describe in both cases limiting L\' evy measures. We will
 give a direct proof there unlike Corollary 1.4 of \cite{TK} which follows from a more
general result whose proof relies on the point processes machinery.
 
 Observe that the one dimensional version of Proposition \ref{prop5.1} is
 applicable to each sum $S_N(\te,t),\,\te\in\Te_*$ in (\ref{5.1}). Indeed,
 each $g_\te,\te\in\Te_*$ has the form
 \begin{equation}\label{5.4}
 g_\te(x_1,...,x_\ell)=(\prod_{i=1}^{p_*}x_{j_i}^{\sig_*})f_\te(x_{p_*+1},...,
 x_\ell)
 \end{equation}
 where $J(\te)=(j_1,j_2,...,j_{{p_*}})$ and $f_\te$ is a monomial in 
 complementary to $x_{j_1},...,x_{j_{p_*}}$ variables rised to powers
  lower than $\sig_*$. Then for $n_1\ne n_2$ we can apply Lemma \ref{lem4.3}
  setting there $Y=g_\te(X_{n_1},X_{n_1+1},...,X_{n_1+\ell-1})$,
  $Z=g_\te(X_{n_2},X_{n_2+1},...,X_{n_2+\ell-1})$ and defining $V_2$ as
  the common part of the products $\prod_{i=1}^{p_*}X^{\sig_*}_{n_1+j_i-1}$
  and $\prod_{i=1}^{p_*}X^{\sig_*}_{n_2+j_i-1}$ (taking $V_2=1$ if these
  products do not have a common part). This will yield the condition
  (\ref{5.2}) for the (scalar) sequence $Z_n=g_\te(X_n,X_{n+1},...,
  X_{n+\ell-1})$.
  
  The main problem in application of Proposition \ref{prop5.1} to the
  vector sums $\{ S_N(\te,t),\,\te\in\Te\}$ is that, in general, we may have pairs
   of large (vector) summands there so that the condition (\ref{5.2}) will
   not hold true. Indeed, consider a simple example with $F(x_1,x_2)=x_1+x_2$
   and $q_1(n)=n,\, q_2(n)=n+1$ so that $\te_1=(1,0),\,\te_2=(0,1)$ and
   $g_{\te_i}(x_1,x_2)=x_i,\, i=1,2$. Then both vectors $Z_n=
   (g_{\te_1}(X_n,X_{n+1}))=(X_n,X_{n+1})$ and $Z_{n+1}=(X_{n+1},X_{n+2})$ 
   may be large in norm if $|X_{n+1}|$ is large. In other words, 
   the probability
   $P\{ |Z_1|>\del b_N,\, |Z_2|>\del b_N\}$ may be of the same order as 
   $P\{ |X_1|>\del b_N\}$ which is of order $1/N$. We observe that in this
   situation there is no weak convergence of the process $\Up_N$ from 
   (\ref{5.3}) in the $J_1$ topology. Indeed, if this convergence would take
    place then also the process $\xi_N(t)=b_N^{-1}\sum_{1\leq n\leq Nt}
    (X_n+X_{n+1})$  would weakly converge in the $J_1$ topology which is
    false as shown in \cite{AT}.
    
     In order to adjust our sums to requirements of Proposition \ref{prop5.1}
     we will perform a rearrangement procedure which will produce new sums
     $\cS_N(\te,t)$ not much different from $S_N(\te,t)$ where large pairs
     of summands can emerge only with negligible probability. In the above
     simple example we will set $\cZ_n=(X_{n+1},X_{n+1})$ for $n\geq 2$
     and the latter vector sequence will satisfy the conditions of 
     Proposition \ref{prop5.1}. This yields the convergence in the $J_1$ 
     topology of the process $\tilde\Up_N(t)$ obtained from $\Up_N(t)$ by
     replacing $Z_n$ by $\cZ_n$ but the estimate
     \[
     |\sum_{1\leq n\leq Nt}Z_n-\sum_{2\leq n\leq Nt}\cZ_n|\leq |X_1|
     +|X_{[Nt]}|
     \]
     does not enable us to obtain convergence of $\Up_N$ in the $J_1$ 
     topology but only provide weak convergence of finite dimensional 
     distributions. In Appendix we will exhibit an alternative approach which 
will yield directly convergence of finite dimensional distributions without the
 condition (\ref{5.2}) as required in Theorem \ref{thm2.3}(i) and the rearrangement
which is special for the polynomial setup here will not be used there.
     
     In order to deal with the general case we consider the disjoint subsets
     $\Psi_1,\Psi_2,...,\Psi_r$ of $\Te_*$ which are not singletons and
     such that for each pair $\te_1,\te_2\in\Psi_i$ the sets $J(\te_1)=
     (i_1,...,i_{p_*})$ and $J(\te_2)=(j_1,...,j_{p_*})$ have the differences
     $j_l-i_l,\, l=1,...,p_*$ equal to a constant independent of $l$. If
     there are no such (non singleton) subsets of $\Te_*$ then we are in the 
     circumstances of the second part of Theorem \ref{thm2.3} where there is
     no need in a rearrangement applying directly Proposition \ref{prop5.1}
     and this case will be discussed later on. Thus, we assume now that such
     subsets exist. In each $\Psi_i,\, i=1,...,r$ choose $\te_i$ such that 
     $J(\te_i)=(j_{i1},j_{i2},...,j_{ip_*})$ has the maximal first index
     $j_{i1}$ where we set the order $j_{i1}<j_{i2}<...<j_{ip_*}$ in 
     $J(\te_i)$. Then there exist integers $0=a_{i1}<a_{i2}<...<a_{iz_i}
     <j_{i1}$ such that $\Psi_i=\{\te_{i1},\te_{i2},...,\te_{iz_i}\}$
     and $J(\te_{il})=(j_{i1}-a_{il},j_{i2}-a_{il},...,j_{ip_*}-a_{il})$ for
     each $l=1,2,...,z_i$.
     
      The goal of our rearrangement procedure is to
     produce new vector summands $\cY_n(\te),\,\te\in\Te_*$ which will
      replace the summands $Y_n(\te)=g_\te(X_n,X_{n+1},...,X_{n+\ell-1}),\,
      \te\in\Te_*$ in (\ref{5.1}) so that Proposition \ref{prop5.1} could
      be applied to the new sum while the difference between these two sums
      can be controlled.
      
      If $\te$ does not belong to some $\Psi_i$ (in particular, if $\te\in
      \Te\setminus\Te_*$) defined above then we set $\cY_n(\te)=Y_n(\te)$. 
      Now, suppose that $\te=\te_{il}$ for some $1\leq l\leq z_i$. then for 
      all $n\geq\ell$ set
      \begin{equation}\label{5.5}
      \cY_n(\te)=g_{\te}(X_{n+a_{il}},X_{n+a_{il}+1},...,X_{n+a_{il}+\ell-1})=
      Y_{n+a_{il}}(\te).
      \end{equation}
      Let
      \begin{equation}\label{5.6}
      \cS_N(\te,t)\sum_{\ell\leq n\leq Nt}\cY_n(\te).
      \end{equation}
      It is easy to see that
      \begin{equation}\label{5.7}
      |S_N(\te,t)-\cS_N(\te,t)|\leq(\sum_{1\leq n\leq 2\ell}+
      \sum_{Nt\leq n\leq Nt+\ell})|g_\te(X_n,X_{n+1},...,X_{n+\ell-1})|.
      \end{equation}
      Relying on Lemma \ref{lem4.3} it is easy to see that the vector 
      summands $\cZ_n=(\cY_n(\te),\,\te\in\Te)$ satisfy the condition
      (\ref{5.2}) of Proposition \ref{prop5.1} (considered with $\cZ_n$ in
      place of $Z_n$), and so as $N\to\infty$ the processes $\Up_N(t),\, 
      t\in[0,T]$ defined
       by (\ref{5.3}) (again, with $\cZ_n$ in place of $Z_n$) weakly converge
       in the $J_1$ topology on $D([0,T],\, R^m)$ to an $\al_*$-stable vector
       L\' evy
       process $\{\Xi(\te,t),\,\te\in\Te\}$, $t\in [0,T]$. The estimate
       (\ref{5.7}) enable us to conclude from here that all finite dimensional 
       distributions of the vector process $\{\Xi_N(\te,t),\,\te\in\Te\}$,
       $t\in [0,T]$ weakly converge to the corresponding finite dimensional
       distributions of $\{\Xi(\te,t),\,\te\in\Te\}$, $t\in[0,T]$. Clearly,
       $\frac 1{b_N}\{\Xi_N(\te,t),\,\te\in\Te\setminus\Te_*\}$ will converge
       to zero in probability.

     Still, the estimate (\ref{5.7}), in general, does not yield weak 
     convergence of the vector process $\{\Xi_N(\te,\cdot),\,\te\in\Te\}$
     in any of Skorokhod's topologies. Indeed, take $F(x_1,x_2)=x_1-x_2$,
     $Z_n=(X_n,-X_{n+1})$ with $X_j$ satisfying (\ref{2.1a}) and define 
     $\Up_N$ by (\ref{5.3}). The weak convergence of $\Up_N(\cdot)$ as $N\to
     \infty$ would imply the weak convergence of the process $\frac 1b_N
     (X_{[Nt]}-X_1),\, t\in[0,T]$ which does not converge in any of
     Skorokhod's topologies as explained in \cite{AT}.
     
     The above provides the proof of the assertion (i) of Theorem 
     \ref{thm2.3}. In order to derive the assertion (ii) we observe that
      since there exist no non singleton subsets $\Psi_i$ satisfying 
      conditions of the above proof it follows that already the vector 
      summands $Z_n=(Y_n(\te),\,\te\in\Te)$ satisfy (\ref{5.2}) of
      Proposition \ref{prop5.1}, and so the vector process
      $\Up_N(t)=\{\Xi_N(\te,t),\,\te\in\Te\}$, $t\in[0,T]$ weakly converges
      in the $J_1$-topology to an $\al_*$-stable L\' evy process $\{\Xi(\te,t),
      \,\te\in\Te\}$, $t\in[0,T]$. 
      
      In order to prove that the component processes $\Xi(\te,\cdot)$ will 
      be independent for different $\te$ it suffices to show that the 
      L\' evy measure of the vector process $\{\Xi(\te,\cdot),\,\te\in\Te\}$
      will be concentrated on axes. Relying on Section 4 from \cite{TK} we
      conclude that the latter follows if the random vector
      $Z=(Y_1(\te),\,\te\in\Te)$ satisfies (\ref{4.4}) with the L\' evy
      measure $\nu$ supported on the axes, i.e.
      \[
      \nu\{ (x_1,...,x_2):\, |x_i|>0\,\,\mbox{and}\,\,|x_j|>0,\, i\ne j\}=0.
      \]
      The latter will hold true if we show that for any $\te,\tilde\te\in\Te$,
      $\te\ne\tilde\te$ and $\del>0$,
      \begin{equation}\label{5.8}
      \lim_{\rho\to\infty}\rho^{\al_*}(\ln\rho)^{-k_*}P\{|Y_1(\te)|>\rho\,\,
      \mbox{and}\,\,|Y_1(\tilde\te)|>\rho\}=0.
      \end{equation}
      If $\te,\tilde\te\in\Te_*$ and there exists no integer $r$ such that
      $J(\tilde\te)=J(\te)+r$ (in the sense of Theorem \ref{thm2.3}(ii))
      then we can represent $Y_1(\te)=V_1V_2$ and $Y_1(\tilde\te)=V_2V_3$
      where $V_1,V_2$ and $V_3$ satisfy conditions of Lemma \ref{lem4.3}
      and then (\ref{5.8}) follows from there. If, say, $\te\in\Te\setminus
      \Te_*$ then (\ref{5.8}) still holds true since 
      \[
      P\{|Y_1(\te)|>\rho\,\, \mbox{and}\,\,|Y_1(\tilde\te)|>\rho\}\leq
      P\{|Y_1(\te)|>\rho\}
      \]
      taking into account that for $\te\in\Te\setminus\Te_*$ the random 
      variables $Y_1(\te)$ have faster decaying tail probabilities than
      for $\te\in\Te_*$.

      The assertion
      (iii) follows from (i) and (ii) taking into account properties of
      weak convergence in the $J_1$ topology.  Namely, regarding weak 
      convergence of finite dimensional distributions of vector processes
      it is clear that this remains true also for sums of components of
      vectors. Furthermore, if a sequence of $R^d$ curves converges in
      the $J_1$-topology then the corresponding time changes are the same
       for all $d$-components, and so the same time changes work also for
        the sums of components of these curves.   \qed

\section{Limit theorem in the arithmetic progression case}\label{sec6}
  \setcounter{equation}{0}

Now we turn to the situation where
\begin{eqnarray*}
 &S_N(\te,t)=\sum_{1\le n\le Nt}Y_n(\te),\, Y_n(\te)=g_\te(X_n,X_{2n},...,
 X_{\ell n}),\, Y_n=\sum_{\te\in\Te}h_\te Y_n(\te)\\
&\mbox{and}\,\, S_N(t)=\sum_{1\le n\le Nt}Y_n=
\sum_{1\le n\le Nt}F(X_n,X_{2n},\ldots, X_{\ell n}).
 \end{eqnarray*} 
As in \cite{KV2} we consider all primes $p_1,\ldots, p_s$  in $1,2,\ldots,
\ell$.  Denote by $\Gamma_1$ the set of numbers $\{p_1^{b_1}\cdots p_s^{b_s}\}$
 with  $ b_1,\ldots,b_s\ge 0$ and arrange them in the increasing order as 
 $1=n_1<n_2 <\cdots<n_q<\cdots$. Let ${\bbZ}_0$ be the set of all positive 
 integers that do not 
 have $p_1,p_2,\ldots, p_s$ as factors. Then any positive integer 
 $n\in {\bbZ}_+$
can be written  uniquely as a product $n=in_q$ where $i\in {\bbZ}_0$ and $n_q
\in\Gamma_1$. It is not difficult to see (for instance, by the 
 inclusion-exclusion principle) that
 the set ${\bbZ}_0$ has density $\rho=\Pi_{i=1}^s (1-\frac{1}{p_i})$
  as a subset of ${\bbZ}_+$. For $i\in {\bbZ}_0 $,  we denote by $\Gamma_i$ the 
  set    $ \{in\}$  with $n$ from $\Gamma_1$.  Then  $\Gamma_i$ are disjoint 
  and ${\bbZ}_+=\cup_{i\in {\bbZ}_0} \Gamma_i$.  If $n\in \Gamma_i $ so is $rn$ 
  for $r\le \ell$. In particular $Y_m$  depends only on the variables $\{X_n\}$
   with $n$ in the same $ \Gamma_i$ to which $m$ belongs and the collections 
   $\{Y_m:\, m\in \Gamma_i\}$ are mutually independent for different values 
   of $i$.  We need the following fact.
   
\begin{lemma}\label{lem6.1}
Given $\ell$, the set ${\bbZ}_+$ of positive integers can be divided into 
$\ell^2+1$ mutually disjoint sets $E_i$ such that
if $r$ and $s$ are two different integers in the same $E_i$  then the two sets 
$\{r,2r,\ldots,\ell r\}$ and $\{s,2r,\ldots,\ell s\}$ are disjoint. 
In particular for every $i$, $\{Y_r: r\in E_i\}$ are mutually independent.
\end{lemma}
 \begin{proof}
Let us construct a graph with ${\bbZ}_+$ as vertices. There is an edge connecting $r$ and $s$, if the two sets $\{r,2r,\ldots, \ell r\}$ and 
 $\{s,2s,\ldots, \ell s\}$ have a common integer. In other words
 $is=jr$ for some $i,j$  between $1$ and $\ell$. For each fixed $r$ the
 equation $\frac{r}{s}=\frac{i}{j}$ has  at most $\ell^2$ solutions in $i$ 
 and $j$. Hence, the vertices  of the graph have 
 degree at most $\ell^2$. Therefore $\ell^2+1$ colors are  enough to make sure 
   that no two vertices connected by an edge have the same color.
 \end{proof}
 
 \begin{lemma}\label{lem6.2}
Let $T<\infty$ be fixed. For each $q>1$ and $N$, let  ${\bf A}_N^q$ be a 
subset of integers from $[1,NT]$ with density $\lim_{N\to\infty}
\frac{|{\bf A}_N^q|}{NT}= \rho_q$ such that $\rho_q\to 0$ as $q\to\infty$. 
Then for any $\te\in\Te$,
$$
\limsup_{q\to\infty}\limsup_{N\to\infty}P\big\{\sup_{0\le t\le T}\frac{1}{b_N}
|\sum_{r\in {\bf A}_N^q,\,r\le Nt} (Y_r(\te)-a_N^\te)| \ge \delta\big\}=0
$$
where 
\[
a_N^\te=E\big(\frac {b_N^2Y_1(\te)}{b^2_N+Y^2_1(\te)}\big).
\]
\end{lemma}
\begin{proof}
We can use Lemma \ref{lem6.1} to split the set   ${\bf A}_N^q$ into
 $\ell^2+1$ subsets such that when $r$ runs within each subset the random
 variables $Y_r(\te)$
 are mutually independent. Since it is enough to prove the estimate for each 
 subset, we can  assume that $Y_r(\te)$ are mutually independent for $r\in 
{\bf A}_N^q$. For each $Y_r(\te)$ with $\te\in\Te_*$ we have the tail estimate,
$$
\lim_{\rho\to\infty} \rho^{\alpha*} (\ln \rho)^{-k_*}E[W(\frac{Y_r}{T})]=\int
  W(x) M(dx)
$$
where
\[
M(dx)=(c_-{\bf 1}_{x<0}+c_+{\bf 1}_{x>0})\frac{dx}{|x|^{1+\frac{\al}{\al_*}}}.
\]
Let $N_q$ be the cardinality of $A^q_N\cap[1,NT]$. If $N_q$ stays bounded
as $N\to\infty$ then the assertion of the lemma is clear. Suppose that
$N_q\to\infty$ as $N\to\infty$. Then relying on standard (functional) stable
limit theorems for i.i.d. random variables (see, for instance, \cite{AG},
\cite{Bi1}, \cite{GK} and \cite{JS}) and properties of the $J_1$ convergence
we conclude that as $N\to\infty$ the quantity  $\sup_{0\le t\le T}
\frac{1}{b_{N_q}}|\sum_{r\in {\bf A}_N^q,\,r\le Nt}(Y_r-a_{N_q})|$ will have
 a limiting distribution which does not depend on $q$. Now, the lemma follows 
 from the observation that
$$
\lim_{q\to\infty}\lim_{N\to\infty}\frac{b_{N_q}}{b_N}=0.
$$
\end{proof}

For each integer $q\geq 1$ set
\[
\Xi_N^q(\te,t)=\frac 1{b_N}\sum_{j=1}^q\, \sum_{i\in {\bbZ}_0,\, in_j\le Nt}
(Y_{in_j}(\te)-a^\te_N).
\]
We will need the following result.

\begin{lemma}\label{lem6.3} For each fixed $T<\infty$ and $\ve>0$,
\begin{equation}\label{6.1}
\lim_{q\to\infty}\limsup_{N\to\infty}P\{\sup_{0\le t\le T}|\Xi_N^q (\te,t)-
\Xi_N(\te,t)|\ge\ve\}=0.  
\end{equation}
\end{lemma}
\begin{proof}
Since 
$$
\Xi_N(\te,t)=\frac{1}{b_N}\sum_{s=1}^\infty\, \sum_{ i\in{\bbZ}_0,\, in_s\le Nt}
 (Y_{in_s}(\te)-a_N^\te)
$$
we can write
$$
\Xi_N(\te,t)-\Xi_N^q(\te,t)=\frac{1}{b_N}\sum_{s=q+1}^\infty 
\sum_{ i\in{\bbZ}_0,\, in_s \le Nt} (Y_{in_s}(\te)-a_N^\te).
$$
Hence,
\begin{align}
\sup_{0\le t\le T}|\Xi_N(\te,t)-\Xi_N^q(\te,t)|&\le \sup_{1\le n\le NT} 
|\sum_{s=q+1}^\infty\, \sum_{ i\in {\bbZ}_0,\, in_s \le n} (Y_{in_s}(\te)-
a_N^\te)|\notag\\
&=\sup_{1\le n\le NT}\frac{1}{b_N}|\sum_{r\in {\bf A}_N^q,\, r\le n} 
(Y_r(\te)-a_N^\te)|\label{6.2}
\end{align}
where ${\bf A}_N^q$ is the set of integers in $[1, NT]$ that are divisible by 
some $n_{q'}$ with $q'>q$. The set  of such integers in the interval  $[1, NT]$
 will have a proportion at most $\sum_{q'>q}\frac{1}{n_{q'}}=\epsilon_q\to 0$
  as $q\to\infty$ since
$$
\sum_q \frac{1}{n_q}=\sum_{b_1,\ldots, b_s\ge 0} \frac{1}{p_1^{b_1}\cdots 
p_s^{b_s}}= \prod_{j=1}^s (1-\frac{1}{p_j})^{-1}=\frac{1}{\rho}<\infty.
$$
Now, the result follows from Lemma \ref{lem6.2}.
\end{proof}

We will study first the limiting behavior of $\Xi^q_N$ as $N\to\infty$ in
the $J_1$ topology and then, relying on Lemma \ref{lem6.3}, will let 
$q\to\infty$ and obtain weak limits of distributions of $\Xi_N$ required in
 Theorem \ref{thm2.4}. 
Set $Z_i^{ j ,\theta}=Y_{in_j}(\te)$ so that $Y_{in_j}=\sum_\theta Z_i^{ j ,
\theta}$.
Recalling the notation $\Theta_*=\{\theta: \al(\theta)=\al_*, k(\theta)=k_*\}$
 from Section \ref{sec2} we define an equivalence relation in $\Gamma_1\times
  \Theta_*$ by declaring $(n_{j_1},\theta_1)\sim(n_{j_2},\theta_2)$ if 
  $n_{j_1}J(\theta_1)$ and $n_{j_2} J(\theta_2)$ are identical as subsets of 
  ${\bbZ}$ (viewing these as products of a scalar and a vector). 
  The equivalence classes will be 
   denoted by $\tau$ and together they form a set $\mathcal T$,  a quotient of
    $\Gamma_1\times\Theta_*$. The joint distribution  of the collection  
    $\{Z_i^{ j ,\theta}\}$ as $j$ and $\theta$ vary does not depend on $i$. 
    Let $\cD_q$ be the (finite) set $\{ (n_j,\te):\, j\leq q,\,\te\in\Te\}$
    whose cardinality we denote by $m_q$. The vector space $R^{m_q}$ can be  
    naturally decomposed as
$$ 
R^{m_q}=\oplus_{\tau\in \mathcal T} V_\tau\oplus U
$$
corresponding to the span of coordinates from the equivalence classes $\tau
\in \mathcal T$ and  the span  $U$ of the remaining coordinates from 
$\mathcal D_q$.
\begin{lemma}\label{lem6.4} Let $\bfZ_q$ be the random $m_q$-dimensional
vector $\{ Z_1^{j,\te},\,j\leq q,\,\te\in\Te\}$. Then the limit
\begin{equation}\label{6.3}
\lim_{\rho\to\infty}\rho^{-\alpha_*}(\ln\rho)^{k_*} EW\big(\frac {\bfZ_q}
{\rho}\big)=\int_{S^{m_q-1}}\int_0^\infty W(su)\nu(ds)\frac{du}{u^{1+\alpha*}}
\end{equation}
 exists for any bounded continuous function $W$ satisfying (\ref{2.1+})
 and for some measure $\nu$ on the $(m_q-1)$-dimensional sphere $S^{m_q-1}$. 
 Moreover, $\nu$ is concentrated on 
 $\cup_{\tau\in \mathcal T}\cdots\{0\}\oplus\{0\}\oplus V_{\tau}\oplus\{0\}
 \oplus \cdots$.  
 \end{lemma} 
 \begin{proof}
 In view of (\ref{5.4}) we can write for any $(n_{j_1},\te_1)$ and
 $(n_{j_2},\te_2)$ with $\te_1,\te_2\in\Te_*$,
 \begin{equation}\label{6.4}
 g_{\te_i}(X_{n_{j_i}},X_{2n_{j_i}},...,X_{\ell n_{j_i}})=\big(\prod_{u=1}^{p_*}
 X^{\sig_*}_{l_u^{(i)}}\big)f_{\te_i}(X_{l^{(i)}_{p_*+1}},...,X_{l^{(i)}_\ell}),
 \, i=1,2
 \end{equation}
 where $n_{j_i}J(\te_i)=(l_1^{(i)},...,l_{p_*}^{(i)})$ while $f_{\te_i}$ is
 a monomial containing $X_{l_v}$'s with indexes $l_v$ in $n_{j_i},2n_{j_i},
 ...,\ell n_{j_i}$ which are different from $l^{(i)}_1,...,l^{(i)}_{p_*}$ 
 and $X_{l_v}$'s
 are rised in $f_{\te_i}$ to powers smaller than $\sig_*$. If $(n_{j_1},
 \te_1)\sim(n_{j_2},\te_2)$ then $(l_1^{(1)},...,l_{p_*}^{(1)})=
 (l_1^{(2)},...,l_{p_*}^{(2)})$, and so for $(n_j,\te)$ within one 
 equivalence class we are in circumstances of Lemma \ref{lem4.2}. If
 $(n_{j_1},\te_1)$ and $(n_{j_2},\te_2)$ are in different equivalence classes
 then $(l_1^{(1)},...,l_{p_*}^{(1)})\ne (l_1^{(2)},...,l_{p_*}^{(2)})$ and
 we can apply Lemma \ref{lem4.3} with $Y=g_{\te_1}(X_{n_{j_1}},X_{2n_{j_1}},
 ...,X_{\ell n_{j_1}})$, $Z=g_{\te_2}(X_{n_{j_2}},X_{2n_{j_2}},...,
 X_{\ell n_{j_2}})$ and $V_2$ consisting of the common part of the products
 $\prod_{u=1}^{p_*}X^{\sig_*}_{l_u^{(1)}}$ and $\prod_{u=1}^{p_*}
 X^{\sig_*}_{l_u^{(2)}}$ while setting $V_2=1$ if this common part is empty.
 Thus, we arrive at the circumstances of Lemma \ref{lem4.1} and the current
 lemma follows from there.
 \end{proof}
    
 Now, set
 \[
 \hat\Xi^{j,\te}_N(t)=\frac 1{b_N}\sum_{i\in\bbZ_0,\, i\leq Nt}(Z^{j,\te}_i
 -a^{\te}).   
 \]
 It follows from Lemma \ref{lem6.4} and standard stable limit theorems for
 sums of i.i.d. (regularly varying) random vectors (see, for instance, 
 \cite{Rv} and Section 7.2 in \cite{Re}) that for each $q<\infty$ as 
 $N\to\infty$ the vector process
 \[
 \hat\Xi^q_N(t)=\{\hat\Xi_N^{j,\te}(t),\, j\leq q,\,\te\in\Te_*\},\, t\in[0,T]
 \]
 weakly converges in the $J_1$-topology to an $\al_*$-stable vector L\' evy
 process
 \[
 \hat\Xi^q(t)=\{\hat\Xi^{j,\te}(t),\, j\leq q,\,\te\in\Te_*\},\, t\in[0,T]
 \]
 while for $\te\in\Te\setminus\Te_*$ the components $\hat\Xi_N^{j,\te}$
 converge to $\hat\Xi^{j,\te}\equiv 0$. Moreover, we obtain also from the
 last assertion of Lemma \ref{lem6.4} that the vector processes
  $\{\hat\Xi^{j,\te},\, (n_j,\te)\in\tau\}$ parametrized by $\tau\in\cT$ are
  mutually independent.
  
  Define the vector processes
  \[
  \tilde\Xi^q_N(t)=\{\hat\Xi^{j,\te}_N(\frac t{n_j}),\, j\leq q,\,
  \te\in\Te_*\},\, t\in[0,T].
  \]
  Since the processes $\hat\Xi_N^{j,\te}$ weakly converge as $N\to\infty$
  in $J_1$-topology
  to the corresponding processes $\hat\Xi^{j,\te}$ then the processes
  $\tilde\Xi^{j,\te}_N( t)=\hat\Xi^{j,\te}_N(\frac t{n_j})$ weakly converge
  as $N\to\infty$
  to $\tilde\Xi^{j,\te}( t)=\hat\Xi^{j,\te}(\frac t{n_j})$. Moreover, observe
  that for a fixed $j$ all processes $\tilde\Xi_N^{j,\te},\,\te\in\Te_*$ are 
  obtained from the corresponding processes  $\hat\Xi^{j,\te}_N,\,\te\in\Te_*$ 
  by the same linear time change, and so we can use for them the same change 
  of time functions appearing in the definition of the $J_1$-convergence. It
  follows that for each $j$ the whole vector process $\tilde\Xi_N^q(j,t)=
  \{\tilde\Xi_N^{j,\te}(t),\, \te\in\Te_*\},\, t\in[0,T]$ converges weakly in
  the $J_1$-topology to an $\al_*$-stable L\' evy vector process 
  $\tilde\Xi^q(j,t)=\{\tilde\Xi^{j,\te}(t),\, \te\in\Te_*\},\, t\in[0,T]$.

  Next, relying on Corollary \ref{cor7.1} in Appendix we will prove that as 
  $N\to\infty$ the full vector process $\tilde\Xi^q_N(t),\, t\in[0,T]$ weakly 
  converges in the $J_1$-topology to the vector process
  \[
  \tilde\Xi^q(t)=\{\tilde\Xi^{j,\te}(t),\, j\leq q,\,\te\in\Te_*\},
  \, t\in[0,T].
 \]
 In order to do this we have to show that with probability one the 
 vector components $\tilde\Xi^q(j,t),\, t\in[0,T]$, $j\leq q$ 
 have no (pairwise) simultaneous jumps. Namely, consider $\tilde\Xi^q(i,t)$
 and $\tilde\Xi^q(j,t)$ for $j>i$. Set $c=\frac {n_j}{n_i}$. For each
 integer $k\geq 0$ and $t\in[n_jc^k,n_jc^{k+1})$ define new vector processes
 \[
 \Psi_k^q(i,t)=\tilde\Xi^q(i,t)-\tilde\Xi^q(i,c^{k+1})\,\,\mbox{and}\,\,
 \Psi_k^q(j,t)=\tilde\Xi^q(j,t)-\tilde\Xi^q(j,c^k).
 \]
 Since the limiting vector L\' evy process $\hat\Xi^q$ has independent 
 increments we obtain that the vector processes $\Psi_k^q(i,t)$ and
 $\Psi_k^q(j,t)$ are independent L\' evy processes when $t\in[n_jc^k,
 n_jc^{k+1}]$, and so almost surely they cannot have simultaneous jumps.
  Hence, with probability one the processes 
 $\tilde\Xi^q(i,t)$ and $\tilde\Xi^q(j,t)$ have no simultaneous jumps
 when $t$ runs in $[0,T]$. 
 
 Now, from the $J_1$-convergence of the full vector process $\tilde\Xi^q_N$ 
 to $\tilde\Xi^q$ we obtain also by Corollary \ref{cor7.1} that the vector
  process $\Xi^q_N=\sum_{j=1}^q\tilde\Xi_N^q(j,\cdot)=\{\Xi^q_N(\te,\cdot),
  \,\te\in\Te_*\}$ weakly converges in the $J_1$ topology to the vector 
  process $\Xi^q=\sum_{j=1}^q\tilde\Xi^q(j,\cdot)=\{\Xi^q(\te,\cdot),
  \,\te\in\Te_*\}$ where $\Xi^q(\te,\cdot)=\sum_{j=1}^q\tilde\Xi^{j,\te}$.
  It follows from the convergence of vector processes that as $N\to\infty$ 
  the sum process $\xi^q_N(t)=\sum_{\te\in\Te}h_\te\Xi_N^q(\te,t),\, t\in[0,T]$
  also converges weakly in the $J_1$ topology to $\xi^q(t)=\sum_{\te\in\Te}
  h_\te\Xi^q(\te,t),\, t\in[0,T]$.
  
   Next, we can write
  \begin{eqnarray}\label{6.5}
 & \Xi^q(\te,t)=\sum_{j=1}^q\hat\Xi^{j,\te}(\frac t{n_j})\\
 &=\sum_{j=1}^q\hat\Xi^{j,\te}(\frac t{n_q})+\sum_{j=1}^{q-1}\sum_{i=1}^j
 (\hat\Xi^{i,\te}(\frac t{n_j})-\Xi^{i,\te}(\frac t{n_{j+1}})).
 \nonumber\end{eqnarray}
 Since for each fixed $\te$ the pairs $(i,\te)$ belong for different $i$'s 
 to different equivalence classes and taking into account that each limiting
 L\' evy process $\hat\Xi^{j,\te}$ has independent increments we conclude
 that the summands in the right hand side of (\ref{6.5}) are independent.
 Hence, the process $\Xi^q(\te,t),\,\te\in\Te_*$ is an $\al_*$-stable L\' evy 
 process. If all equivalence classes are singletons then the whole limiting 
 vector process $\hat\Xi^q$ has independent increments and this remains true 
 for $\Xi^q$, as well, in view of its construction by taking sums of certain
  components of $\hat\Xi^q$.
  
  It remains to let $q\to\infty$ and to verify properties of corresponding
  limiting processes. Denote by $\cL^{q,\te}_N,\,\cL^{q,\te},\,\cL^{q}_N,\,
  \cL^{q},\,\hat\cL^{\te}_N$ and $\cL_N$ the distributions of processes
  $\Xi^q_N(\te,\cdot),\,\Xi^q(\te,\cdot),\,\Xi^q_N,\,\Xi^q,\,
  \Xi_N(\te,\cdot)$ and $\{\Xi_N(\te,\cdot),\,\te\in\Te_*\}$, respectively,
  on the time interval $[0,T]$. Denote by $d$ the Prokhorov distance (see,
  for instance, \cite{Bi1} or \cite{JS}) on the corresponding space of
  distributions and observe that convergence in this metric is equivalent
  here to the weak convergence with respect to the $J_1$ topology. The above
  proof yields that for each $\te\in\Te$ and $q\geq 1$,
  \begin{equation}\label{6.6}
  \lim_{N\to\infty}d(\cL_N^{q,\te},\cL^{q,\te})=\lim_{N\to\infty}d(\cL^q_N,
  \cL^q)=0.
  \end{equation}
  It is also clear from the corresponding definitions that for each $\te\in\Te$
  and $N\geq 1$,
   \begin{equation}\label{6.7}
  \lim_{q\to\infty}d(\cL_N^{q,\te},\hat\cL^{\te}_N)=\lim_{q\to\infty}
  d(\cL^q_N,\cL_N)=0.
  \end{equation}
 In addition, it follows from Lemma \ref{lem6.3} that for each $\te\in\Te$,
 \begin{equation}\label{6.8}
  \lim_{q\to\infty}\limsup_{N\to\infty}d(\cL_N^{q,\te},\hat\cL_N^{\te})=
  \lim_{q\to\infty}\limsup_{N\to\infty}d(\cL^q_N,\cL_N)=0.
  \end{equation}
  By the triangle inequality for any $q,q'\geq 1$,
  \[
  d(\cL^{q,\te},\cL^{q',\te})\leq d(\cL^{q,\te},\cL^{q,\te}_N)+d(\cL^{q,\te}_N,
  \hat\cL_N^{\te})+d(\cL^{q',\te}_N,\hat\cL^{\te}_N)+d(\cL^{q',\te},
  \cL^{q',\te}_N)
  \]
  and 
  \[
  d(\cL^{q},\cL^{q'})\leq d(\cL^{q},\cL^{q}_N)+d(\cL^{q}_N,
  \cL_N)+d(\cL^{q'}_N,\cL_N)+d(\cL^{q'},\cL^{q'}_N).
  \]
  These together with (\ref{6.6})--(\ref{6.8}) yields that
 \begin{equation}\label{6.9}
  \lim_{q,q'\to\infty}d(\cL^{q,\te},\cL^{q',\te})=\lim_{q,q'\to\infty}
  d(\cL^q,\cL^{q'})=0.
  \end{equation}
  Hence, $\{\cL^{q,\te},\, q\geq 1\}$ and $\{\cL^q,\, q\geq 1\}$ are Cauchy
  sequences in the corresponding complete metric spaces, and so there exist
  distributions $\hat\cL^\te$ and $\cL$ on the corresponding spaces such that
  \begin{equation}\label{6.10}
  \lim_{q\to\infty}d(\cL^{q,\te},\hat\cL^{\te})=\lim_{q\to\infty}d(\cL^q,
  \cL)=0.
  \end{equation}
 It follows also from (\ref{6.6})--(\ref{6.8}) that
 \begin{equation}\label{6.11}
  \lim_{N\to\infty}d(\hat\cL_N^{\te},\hat\cL^{\te})=\lim_{N\to\infty}d(\cL_N,
  \cL)=0.
  \end{equation}
  
  Let also $\tilde\cL^q_N,\,\tilde\cL^q$ and $\tilde\cL_N$ be distributions of
  the sum processes $\xi^q_N(\cdot)=\sum_{\te\in\Te}h_\te\Xi^q_N(\te,\cdot)$,
   $\xi^q(\cdot)=\sum_{\te\in\Te}h_\te\Xi^q(\te,\cdot)$ and 
   $\xi_N(\cdot)=\sum_{\te\in\Te}h_\te\Xi_N(\te,\cdot)$, respectively. Then,
   in the same way as above we obtain that there exists a distribution
   $\tilde\cL$ such that
   \begin{equation}\label{6.12}
  \lim_{q\to\infty}d(\tilde\cL^q,\tilde\cL)=\lim_{N\to\infty}d(\tilde\cL_N,
  \tilde\cL)=0. 
  \end{equation}
  
  Let $\Xi(\te,\cdot),\,\{\Xi(\te,\cdot),\,\te\in\Te_*\}$ and $\xi$ be the 
  processes on the time interval $[0,T]$ having the distributions $\hat\cL^\te,
  \,\cL$ and $\tilde\cL$, respectively. Thus, we established the following 
  weak convergencies in the $J_1$ topology as $N\to\infty$,
  \[
  \Xi_N(\te,\cdot)\Rightarrow\Xi(\te,\cdot),\,\{\Xi_N(\te,\cdot),\,\te\in\Te_*\}
  \Rightarrow\{\Xi(\te,\cdot),\,\te\in\Te_*\}\,\,\mbox{and}\,\,\xi_N
  \Rightarrow\xi.
  \]
  Observe that from convergence of the vector process above it follows that
  $\xi(\cdot)=\sum_{\te\in\Te}h_\te\Xi(\te,\cdot)$. 
  
  Since for each $q$ the processes $\Xi^q(\te,\cdot)$, $\{\Xi^q(\te,\cdot),
  \,\te\in\Te_*\}$ and $\xi^q$ have $\al^*$-stable distributions we derive from
  convergence as $q\to\infty$ of characteristic functions of marginal and 
  finite dimensional distributions of these processes that the limiting 
  processes $\Xi(\te,\cdot)$, $\{\Xi(\te,\cdot),\,\te\in\Te_*\}$ and $\xi$ 
  must have $\al^*$-stable distributions, as well. Under the conditions of 
  Theorem \ref{thm2.4}(ii) for each $q$ the processes $\Xi^q(\te,\cdot)$, 
  $\{\Xi^q(\te,\cdot),\,\te\in\Te_*\}$ and $\xi^q$ have independent increments
   and, again, in view of convergence of characteristic functions we see that
   the limiting as $q\to\infty$ processes $\Xi(\te,\cdot)$, $\{\Xi(\te,\cdot),\,
   \te\in\Te_*\}$ and $\xi$ have independent increments, as well, i.e. they are
   L\' evy processes, completing the proof of Theorem \ref{thm2.4}.   \qed

\section{Appendix} \label{secA}
\setcounter{equation}{0}

We will start here with a basic property of the $J_1$-convergence.
Let $X$ and $Y$ be two Polish spaces and $D([0,1];X)$ and $D([0,1];Y)$ are 
spaces of $X$ valued and $Y$ valued functions 
that are right continuous, have left limits at every $t\in [0,1]$ and are,
in addition, left continuous at $1$.
In general, if $x_n(\cdot)\in D([0,1];X)$ and $y_n(\cdot)\in D([0,1];Y)$
converge in $J_1$ topology respectively to $x(t)$ and $y(t)$ then
it does not follow that  $z_n(\cdot)=(x_n(\cdot),y_n(\cdot))$ converges to 
$z(\cdot)=(x(\cdot),y(\cdot))\in D([0,1];Z)$ where $Z=X\times Y$. It is 
possible that $x_n(\cdot)$ and $y_n(\cdot)$ have jumps at two distinct points 
 $t_n, t^\prime_n$  that tend to a common point $t$. As functions with values
 in $Z=X\times Y$, $z_n(t)$ has two jumps that come together and this rules 
 out convergence in $D([0,1];Z)$. However we have the following 
 
\begin{lemma}\label{lem7.0}
Let $x_n(\cdot)$ and $y_n(\cdot)$ converge in $J_1$ topology to $x(\cdot)$ and
 $y(\cdot)$ respectively in $D([0,1];X)$ and $D([0,1];Y)$ . Let $x(\cdot)$ and
  $y(\cdot)$ do not have a jump at the same $t$, i.e the sets of discontinuity 
  points of $x(\cdot)$ and $y(\cdot)$ are disjoint. Then $z_n(\cdot)=
  (x_n(\cdot),y_n(\cdot))$ converges in $J_1$ topology to $z(\cdot)=(x(\cdot),
  y(\cdot))$ in $D([0,1]; Z)$ where $Z=X\times Y$. In particular if $f:Z\to S$
   is a continuous map then $f(x_n(\cdot),y_n(\cdot))$ converges in the  $J_1$
    topology to $f((x(\cdot),y(\cdot))\in D([0,1];S)$ . 
\end{lemma}
\begin{proof}
Since $x_n(t)$ and $y_n(t)$ converge in the $J_1$ topologies on $X$ and $Y$,
respectively, there are compact
 sets $K_X$ and $K_Y$ in $X$ and $Y$ such that
$x_n(t)\in K_X$, and $y_n(t)\in K_Y$  for all $n$ and $t\in [0,1]$. Therefore
 $z_n(t)\in K=K_X\times K_Y$ for all $n$ and $t\in [0,1]$. We need to control
  uniformly the $D[0,1]$ modulus of continuity $\omega_h(z_n(\cdot))$   of 
  $z_n(t)=(x_n(t),y_n(t))$ where
$$
\omega_h(z(\cdot))=\sup_{t_1,t_2: |t_1-t_2|\le h}\inf_{\tau\in (t_1,t_2)} 
\max [\Delta_{(t_1,\tau)}(z(\cdot)),\Delta_{(\tau,t_2)}(z(\cdot))]
$$
and 
$$\Delta_{(a,b)}(z(\cdot))=\sup_{t,s\in (a,b)} d(z(t),z(s))$$
is the oscillation of $z(\cdot)$ in the interval $(a,b)$. We can take
$d(z_1,z_2)=d_1(x_1,x_2)+d_2(y_1,y_2)$ where $d, d_1, d_2$ are the metrics in
 $Z,X,Y$ respectively. Then with the obvious definitions of $\Delta_{(a,b)}
 (x(\cdot))$ and $\Delta_{(a,b)}(y(\cdot))$,
\begin{align*}
\omega_h(z(\cdot))\le &\sup_{t_1,t_2: |t_1-t_2|\le h}\inf_{\tau\in (t_1,t_2)} 
\max [\Delta_{(t_1,\tau)}(x(\cdot)),\Delta_{(\tau,t_2)}(x(\cdot))]\\
&\qquad+ \sup_{t_1,t_2: |t_1-t_2|\le h}\sup_{\tau\in (t_1,t_2)} \max 
[\Delta_{(t_1,\tau)}(y(\cdot)),\Delta_{(\tau,t_2)}(y(\cdot))].
\end{align*}
The convergence of $x_n(\cdot)$ and $y_n(\cdot)$ in the corresponding $J_1$
 topologies guarantees that
\begin{align*}
\lim_{h\to 0}&\limsup_{n\to \infty}\bigg[\sup_{t_1,t_2: |t_1-t_2|\le h}
\inf_{\tau\in (t_1,t_2)} \max [\Delta_{(t_1,\tau)}(x_n(\cdot)),
\Delta_{(\tau,t_2)}(x_n(\cdot))]\\
&\qquad+ \sup_{t_1,t_2: |t_1-t_2|\le h}\inf_{\tau\in (t_1,t_2)} 
\max [\Delta_{(t_1,\tau)}(y_n(\cdot)),\Delta_{(\tau,t_2)}(y_n(\cdot))]\bigg]=0.
\end{align*}
Since the jumps of
$x_n(\cdot)$ and $y_n(\cdot)$ converge individually to the jumps of $x(\cdot)$ 
and $y(\cdot)$ while $x(\cdot)$ and $y(\cdot)$ do not have any  common jumps,
for any $\epsilon >0$ there is  a $\delta>0$ such that  all  the jumps of 
$x_n(\cdot)$ and $y_n(\cdot)$ of size at least $\epsilon>0$ are uniformly   
separated from one another by some $\delta=\delta(\epsilon)>0$ . We can now 
estimate the $D[0,1]$  modulus continuity of  $z_n(\cdot)$. 
 If $h<\delta$ any interval of length $h$ will have at most one 
jump of size larger than $\epsilon$. Therefore of the two components  
$x_n(\cdot)$ and $y_n(\cdot)$ only one of them can have a jump larger than 
$\epsilon$. If  $y_n(t)$  does not have a jump of size larger than $\epsilon$
  in $(t_1,t_2)$ and $|t_2-t_1|<h$  then
\begin{eqnarray*}
&\sup_{\tau\in (t_1,t_2)} \max [\Delta_{(t_1,\tau)}(y_n(\cdot)),\Delta_{(\tau,
t_2)}(y_n(\cdot))]\le \Delta_{(t_1,t_2)}(y_n(\cdot))\\
&\le 2\inf_{\tau\in (t_1,t_2)} \max [\Delta_{(t_1,\tau)}(y_n(\cdot)),
\Delta_{(\tau,t_2)}(y_n(\cdot))]+\epsilon
\end{eqnarray*}
Therefore 
$$
\lim_{h\to 0}\limsup_{n\to\infty} \omega_h(z_n(\cdot))=0
$$
and we are done.
\end{proof}
\begin{corollary}\label{cor7.1}

(i) If $x^i_n$ converge to $x^i$, $i=1,...,d$ in the $J_1$ topology on 
$D([0,1]; X_i)$, where $X_i$ are Polish spaces, and for  any pair $i\not=j$
 the limits $x^i$ and $x^j$ have no common jumps, i.e
the discontinuity points $U_l=\{t: x^l(t-0)\not= x^l(t\})$ for $l=i$ and $l=j$
are disjoint, then as $n\to\infty$ the $d$-vector functions $\{ x_n^i,\, 
i=1,...,d\}$ converge to $\{ x^i,\, i=1,...,d\}$ in the $J_1$ topology of 
the product space $D([0,1];\prod^d_{i=1}X_i)$.

(ii) Let $x_n^i$ and $x^i$ satisfy conditions of (i) with $X_i=R^q,\, 
i=1,...,d$ for some integer $q\geq 1$. Then $\sum_{i=1}^dx_n^i$ converges
in the $J_1$ topology of $D([0,1];R^q)$ to $\sum_{i=1}^dx^i$.

(iii) Let $\{P_n,\, n\geq 1\}$ be probability measures on $D([0,1];
\prod_{i=1}^dX_i)$, where $X_i, i=1,...,d$ are the same as in (i), and
suppose that the marginals on $D([0,1];X_i),\, i=1,...,d$ of $P_n$'s 
converge weakly with respect to the $J_1$ topology while the  joint finite
 dimensional distributions of $P_n$ converge on $D([0,1];\prod_{i=1}^dX_i)$
  to a limit $P$. If  $P$ almost surely the components of $d$-vector functions
  $(x_1(t),...,x_d(t))$ have no common jumps pairwise then $P_n$ weakly
   converges to $P$ as $n\to\infty$ in the $J_1$ topology.
\end{corollary}

While considering real valued  random variables   in the domain of attraction
 of a stable law it is natural to consider     tail behavior of the form
\begin{equation}\label{eq5.1}
\lim_{T\to\infty} T^\al(\ln T)^{-k} P[\pm X\ge T]=c_\pm.
\end{equation}
If $X$ is  $R^d$ valued then  a tail behavior similar to the one dimensional 
 case  above will be to require that for every continuous  function $f$ on the unit
  sphere  $S^{d-1}$ the limit

\begin{equation}\label{eq5.2}
 \lim_{\rho\to\infty}\rho^{\alpha} (\ln \rho)^{-k}E\big[{\bf 1}_{|X|\ge \rho}
f\big(\frac{X}{|X|}\big) \big]=\int_{S^{d-1}} f(s)\nu(ds)
\end{equation}
exists where $\nu$ is a finite nonnegative measure on $S^{d-1}$.  To make the
 connection we need only to think of $S^0$ as
$\pm 1$ and $\nu(\{\pm 1\})=c^\pm$. The following result essentially coincides
with Theorem 3.6 of \cite{Re} but for readers' convenience we provide its
proof here.

\begin{lemma}\label{l5.1} The relation (\ref{eq5.2})  holds true for every
 bounded continuous $f$ on $S^{d-1}$ if and only if  
\begin{equation}\label{eq5.3}
\lim_{\rho\to\infty} \rho^\al (\ln \rho)^{-k} E[W(\frac{X}{\rho})]=
\int_{S^{d-1}}\int_0^\infty  W(sr)\frac{\alpha\nu(ds)\,dr}{r^{1+\alpha}}
\end{equation}
for every $W$ from the space ${\mathcal W}$ of bounded continuous functions 
satisfying (\ref{2.1+}).
\end{lemma}
\begin{proof}
In (\ref{eq5.2}) we can replace $\rho$ by $\rho z$ with $z>0$, to get
 \begin{align}\label{eq5.4}
 \lim_{\rho\to\infty}\rho^{\alpha} (\ln \rho)^{-k}E\big[{\bf 1}_{\frac{|X|}
 {\rho}\ge z}f\big(\frac{X}{|X|}\big) \big]&=\frac{1}{z^\al}\int_{S^{d-1}} 
 f(s)\nu(ds)\notag\\
&=\int_{S^{d-1}} f(s)\nu(ds)\int_z^\infty\frac{\alpha}{u^{1+\alpha}}du
\end{align}
It is now easy to conclude that if $V(r,s)$ is a continuous function of $r>0$ 
and $s\in S^{d-1}$ and   for some $\delta>0$ it is identically $0$ if $r\le 
\delta$ then
 \begin{align}\label{eq5.5}
 \lim_{\rho\to\infty}\rho^{\alpha} (\ln \rho)^{-k}E\big[ V(\frac{|X|}{\rho}, 
 \frac{X}{|X|}) \big]=\int_0^\infty\int_{S^{d-1}} V(u,s)\nu(ds)\frac{\alpha}
 {u^{1+\alpha}}du.
\end{align}
 We now take $V(r,s)=W(rs)$ and obtain (\ref{eq5.3}). To control the 
 contribution near $0$ for $W\in\mathcal W$ we denote by $R(x)$
the tail probability $P[|X|\ge x]$ and obtain
\begin{eqnarray*}
&\rho^\alpha (\ln\rho)^{-k} E[W(\frac{X}{\rho}){\bf 1}_{\frac{|X|}{\rho}
\le \delta }]\le C\rho^\alpha (\ln \rho)^{-k} E[(\frac{X^2}{\rho^2})
{\bf 1}_{\frac{|X|}{\rho}|\le \delta}]\\ 
&=C \rho^{\alpha-2} (\ln \rho)^{-k} E[X^2{\bf 1}_{|X|\le \delta T}]
=-C T^{\alpha-2} (\ln T)^{-k} \int_0^{\delta \rho} x^2 dR(x)\\ 
&\le C \rho^{\alpha-2} (\ln \rho)^{-k}
 \int_0^{\delta \rho} 2\,x\,R(x)\,dx\\ 
&\le C\rho^{\al-2} (\ln \rho)^{-k}\int_0^{\delta\rho}  x (1+x)^{-\alpha}
(\ln (2+x))^kdx\le C\delta^{2-\alpha }
\end{eqnarray*}
is uniformly controlled because $\alpha<2$. Finally to  go from 
(\ref{eq5.3}) to (\ref{eq5.2}), we take $W(x)={\bf 1}_{[1,\infty]}(x) f(\frac{x}
{|x|})$ which can be justified by approximating ${\bf 1}_{[1,\infty]} $ by 
continuous functions. 
\end{proof}

\begin{remark}\label{remA1}
Let  $\{X_n\}$  be a sequence of independent and identically distributed random
 vectors in $R^d$ that satisfy (\ref{eq5.3}).  Let
$$
S_N(t)=\sum_{1\le n\le Nt} X_n
$$
and
\begin{equation}\label{eqA3}
\Xi_N(t)=\frac{1}{b_N}[S_N(t)-Nta_N]
\end{equation}
where the normalizer $b_N$ is given by
$$
b_N=N^{\frac{1}{\al}}(\frac{\ln N}{\al})^{\frac{k}{\al}}
$$
and the centering $a_N$ is given by
$$
a_N=E\big[\frac{b_N^2 X_1}{b_N^2+|X_1|^2}\big].
$$
 Then, according to standard limit theorems for sums of independent random
  vectors (see, for instance, \cite{Rv} and Section 7.2 in \cite{Re}), 
  the processes  $\Xi_N(t)$ converge in the Skorokhod $J_1$ topology 
  on $D[[0,T]; R^d]$ to a limiting stable process $\Xi$ with the 
  characteristic function of the increments $\Xi(t)-\Xi(s)$ given by 
\begin{equation}\label{eqA4}
E[e^{i<\xi,\Xi(t)-\Xi(s)>}]=\exp[(t-s)\psi(\xi)]
\end{equation}
where
\begin{align*}
\psi(\xi)&=\int_{R^d\setminus \{0\}}      [e^{i<\xi, y>}-1- \frac{i<\xi,y>}
{1+|y|^2}] \mu(dy)\\
&=\int_{S^{d-1}}\int_0^\infty  [e^{i<\xi, sr>}-1- \frac{i<\xi,sr>}{1+|r|^2}] 
\frac{\nu(ds) dr}{r^{1+\alpha}}
\end{align*}
It follows that if $\nu$ is concentrated on axes then components of the
process $X(t)$ are independent.
The  proof relies on the calculation 
$$
\lim_{N\to\infty} N\ln [E[e^{i<\xi, X^\prime_N>}]]=\psi(\xi)
$$ 
where $X^\prime_N=\frac{1}{b_N}[X_1-a_N]$.
\end{remark}

\begin{remark}\label{remA2}
The basic assumptions involved in proving the convergence to the  stable 
process with independent  increments is independence of the random variables, 
a common distribution with the correct tail behavior that puts them in the 
domain of attraction of the stable distribution with L\'{e}vy measure  
$\frac{\nu(ds)dr}{r^{1+\alpha}}$ on $R^d$. It is possible that we have a 
random vector in which two components are not independent but in the limiting 
distribution they become independent.
If the L\'{e}vy measure is $\frac{\nu (ds)dr}{r^{1+\alpha}}$, then for  $R^d$ 
 to split  as a sum $\oplus V_j$ with   components of the random vector 
 corresponding to different $V_j$  being  
 mutually independent it is necessary and sufficient that $\nu$ is supported 
 on the union of the subspaces $V_j$. In other words $\nu[ |x_j|>0, |x_{j'}|>0]
 =0 $ where $x=\sum x_j$ is the natural decomposition of $x$
into components from $\{V_j\}$. This requires that
\[
\lim_{\rho\to\infty} \rho^{\al}(\ln\rho)^{-k} P\{|X_{1j}|>\rho, |X_{1j'}|>
\rho\}=0
\]
for any $j\ne j'$ where $X_1=(X_{11},X_{12},...,X_{1d})$ and $X_1$ is the same 
as in Remark \ref{remA1}. 
\end{remark}

\begin{remark}\label{remA3} Let $\{ X_n\}$ be as in Remark \ref{remA1} and
$T$ be a linear map $R^d\to R^m$. Then the process 
$$
Y_N(t)=T\Xi_N(t)=\sum_{1\le n\le Nt} \frac{1}{b_N}(TX_n-Ta_N)
$$
 will, after normalization,  converge to $Y(t)=T\Xi(t)$, a process with 
independent increments given by
$$
E[e^{i<\xi',Y(t)-Y(s)}]=\exp[(t-s)[i <\gamma, \xi'>+\psi'(\xi') ]]
$$ 
where
$$
\psi^\prime(\xi^\prime)=\int_{R^m\setminus \{0\}} [e^{i<\xi', s'r>}-1- 
\frac{i<\xi',sr>}{1+|r|^2}] \frac{\nu'(ds') dr}{r^{1+\alpha}}
$$
and 
$$
\gamma'=\int_{R^d} \bigg[\frac{T(sr)} {1+|T(sr)|^2} - \frac{T(sr)} {1+|r|^2}
\bigg] \frac{\nu(ds) dr}{r^{1+\alpha}}.
$$
The amount by which the process needs centering is only unique up to a  
constant and this requires us to make the adjustment  with  the term $\gamma^
\prime t$ which is the difference between two possibly infinite terms.  
It is not hard 
to see that as $N\to\infty$, $b_N\to\infty$ and by the bounded convergence 
theorem, $\frac{a_N}{b_N}\to 0$.
\end{remark}

In Section \ref{sec5} we proved convergence of the finite dimensional distributions to a stable L\'{e}vy process, by appealing to \cite{TK}.
This required a rearrangement  of the monomial terms, that make
up the polynomial. The new process converged in $J_1$ topology. If condition
 (\ref{A3}) below
were satisfied then the rearrangement was not needed and the original process
 converged  in the $J_1$ topology. We provide here
an alternate proof that does not require this modification, but yields 
directly  finite dimensional convergence  for stationary $\ell$ dependent 
sums as well as convergence in $J_1$ topology under the additional assumption 
(\ref{A3}). 

 \medskip
Let $H$ be  a finite dimensional Euclidean space and  $\{X_i,\, i\geq 1\}$ be 
a stationary sequence of $H$ valued random variables that are $\ell$-dependent 
and have  regularly varying heavy  tails with an index
$\al\in(0,2)$ and a L\' evy measure $M$, i.e. (\ref{eq5.2}) holds true with 
$X=X_1$ and $\nu=M$. We will assume that for $d=2\ell-1$ and all $f\in 
{\mathcal W}(H^d)$,
\begin{equation}\label{A1}
\lim_{N\to\infty} N E[ f(\frac{X_1}{b_N},\ldots, \frac{X_d}{b_N})]=\int_{H^d}
f(x_1,\ldots,x_d) M^d (dx).
\end{equation}
Observe that when the stationary $\ell$-dependent sequence of random vectors
above is obtained in the framework of Theorem \ref{thm2.3} then relying on
Lemma \ref{lem4.1} and Corollary \ref{cor4.3+} we see that the condition
(\ref{A1}) is automatically satisfied.

\begin{lemma} Suppose that (\ref{A1}) holds true.
Then for any integer $k\geq 1$ the limit
\begin{equation}\label{A1+}
\lim_{N\to\infty} N E[ f(\frac{X_1}{b_N},\ldots, \frac{X_k}{b_N})]=\int_{H^d} 
f(x_1,\ldots,x_k) M^k (dx)
\end{equation}
exists and $M^k$ can be computed from $M^d$.
\end{lemma}
\begin{proof}
If for $k>d$ the limits do not exist then because of stationarity  we can 
always select a subsequence such that the limits exist for all values of $k$
 and $f\in {\mathcal W}(H^k)$
and we  will continue to denote  
corresponding limiting measures by $M^k$.   If we can recover $M^k$ from $M^d$, then all subsequences will have the same limit and hence the limit as $N\to\infty$ will exist. For $-\infty<a\le b<\infty$  with  $b-a\le 2\ell-2$, $J=\{i: a\le i\le b\}$ will be of size at most $2\ell-1$ and the limits $M^J$ on $H^J$ will exist and be translation invariant. For  any partition of $J$ into 
disjoint subsets $A$ and $B$ we can write $H^J\backslash \{0\}$ as the disjoint union
$$
H^J\backslash \{0\}=\cup_{ A \in {\mathcal  A}(J)}H^J_A
$$
where ${\mathcal A}(J)=\{ A:\, (A,B)\in{\mathcal P}(J)\}$ and ${\mathcal P}(J)$ is the set of partitions with nonempty $A$ and
$$
H^J_A=\{x\in H^J: x_i\not=0\ \forall \ i\in A;\ x_i=0\ \forall i\in B\}.
$$  
We can write
$M^J=\sum_{A\in{\mathcal A}(J)}M^J_A$, where
$M^J_A$ is the restriction of $M^J$ to $H^J_A$. By the $\ell$ dependence we 
obtain that if 
$|i-j|\ge \ell$ then for any $\delta>0$
$$
\lim_{N\to\infty} N P[|X_i|\ge \delta  b_N, |X_j|\ge \delta b_N ]=0
$$
which in turn implies that $M^J[ x_i\not=0, x_j\not=0]=0$ whenever $i,j\in J$ 
and $|i-j|\ge \ell$.  Any set $A\in\mathcal A(J)$ 
can be ordered $a_1<a_2<\ldots a_r$  and  integers $r$, and $\sigma_i=
a_{i+1}-a_i$ for $i=1,\ldots,r-1$ determine $A$ up to a translation. Set  
$sp(A)=1+\sigma_1+\cdots+\sigma_r=a_r-a_1+1$.
Then $M^J_A=0$ unless $sp(A)\le \ell$. For any $A$ consider the extended 
interval 
${\widehat A}=\{a_r-\ell+1\le i\le a_1+\ell-1\}$. It is the set of  integers 
$j$ such that
$|j-i|\le \ell-1$ for all $i\in A$. In particular $M^J_A$ is determined by  
$M^{\widehat A}_A$ if $J\supset \widehat A$.
We write 
$$
M^J=\sum_{A\in{\mathcal A}(J)}M^J_A=\sum_{A\in{\mathcal A}(J)\atop {\widehat A}
\subset J}M^J_A+\sum_{A\in{\mathcal A}(J)\atop {\widehat A}\not\subset J}M^J_A.
$$
Since $M^J$ determines $M^J_A$ it determines also $M^{\widehat A}_A$ if 
${\widehat A}\subset J$.
If ${\widehat A}\not\subset J$ we can replace $J$ by $I=J\cup {\widehat A}$ 
and 
$$
M^J_A=\sum_{B\supset A} M^I_B=\sum_{B\supset A} M^{\widehat B}_B
$$
Clearly, $M^{\widehat B}_B=0$ unless $sp(B)\le \ell$ and so $sp(\widehat B)
\le 2\ell-1$. By translation invariance all these measures are determined by 
$M^{2\ell-1}$. 
\end{proof}
Sets  with $sp(A)\le \ell$ can be characterized by  $r$  and $a_1<a_2
\ldots<a_r$ with $a_r-a_1\le \ell-1$ and up to a translation by $r$ and 
$\sigma_i\ge 1 $ for $i=1,\ldots ,r-1$ with $\sum_{i=1}^{r-1}\sigma_i\le 1$.
We map $H^{\widehat A}$ into $H$ by $x=\sum_{i\in  {\widehat A}} x_i$ and the
 push forward  of the measure  $M^{\widehat A}_A$  is denoted  by $M^{r,
 \{\sigma_i\}}$ on $H$. Let
\begin{equation}\label{A4}
M^\ast=\sum_{r,\sigma_1,\ldots,\sigma_r} M^{r,\{\sigma_i\}}.
\end{equation}
For each $r,\{\sigma_i\}$ we set 
$$
\gamma_{r,\{\sigma_i\}}=\int_{H^{\widehat A} }\big[\frac{\sum_{i\in A} x_i}
{1+|\sum_{i\in A} x_i|^2}-\sum _{i\in A} \frac{x_i}{1+|x_i|^2}\big]
M^{\widehat A}_A( dx)
$$
and let
$$
\gamma=\sum_{r,\{\sigma_i\}} \gamma_{r,\{\sigma_i\}}.
$$
We note that $\gamma_{1,\{\sigma_i\}}=0$ because $A$ contains only one integer.

\begin{theorem}
Under the condition (\ref{A1}) the  finite dimensional distributions of the 
process
\begin{equation}\label{A2}
\xi_N(t)=\sum_{1\le i\le Nt}\frac{1}{b_N}  [X_i-a_N]
\end{equation}
converge weakly to those of a stable process with the logarithm of its 
characteristic function  at time $t$ given by
\begin{equation}\label{A8}
\log \psi_t(u)=it\langle \gam ,u\rangle+t\int_H [e^{i\langle u,x \rangle}-1-
\frac{i\langle u, x\rangle}{1+\|x\|^2} ]M^\ast(dx)
\end{equation}
where $M^\ast$ is computed from $M^d$ as in (\ref{A4}). In addition, 
if for  all $i\not= j$ 
\begin{equation}\label{A3}
\lim_{N\to\infty} N P[|X_i|\ge \delta  b_N, |X_j|\ge \delta b_N ]=0
\end{equation}
then there is no need in the assumption (\ref{A1}) as (\ref{A1+})
automatically holds true for any integer $k\geq 1$ and the weak convergence 
as $N\to\infty$ of the processes $\xi_N$
takes place in the $J_1$ topology. Furthermore, in the latter case $M^*=M$.
\end{theorem}
Note that since $\frac{a_N}{b_N}\to 0$ as $N\to\infty$, (\ref{A3}) is 
equivalent to
\begin{equation*}
\lim_{N\to\infty} N P[|X_i-a_N|\ge \delta  b_N, |X_j-a_N|\ge \delta b_N ]=0
\end{equation*}
\begin{proof}
 Let $k$ be an integer that will eventually  get large. To exploit $\ell$ 
 dependence we want to sum over blocks of size $k$ and leave gaps of size 
 $\ell-1$.  We divide the set of  positive integers into blocks of size 
 $k+\ell-1$. $B(r)=\{i: (r-1)(k+\ell-1)+1\le i\le (r-1)(k+\ell-1)\}$. Each 
 $B(r)$ consists of the initial segment  $B^+(r)$ of length $k$ and the gap 
 $B^-(r)$ of size $\ell-1$.
$Z^+=\cup_r B^+(r)$ and $Z^-(r)=\cup_r B^-(r)$. 
We define 
 
$$
\xi^k_N(t)=\frac{1}{b_N}\sum_{1\le i\le Nt\atop i\in Z^+  }[X_i-a_N] 
$$
and
$$
\eta^k_N(t)=\frac{1}{b_N}\sum_{1\le i\le Nt\atop i\in Z^-  }[X_i-a_N] 
$$
so that $\xi_N(t)=\xi^k_N(t)+\eta_N^k(t)$. It follows from Lemma \ref{lem6.2} 
considered with
$X_i$'s in place of $Y_i(\theta)$'s (as the proof in this lemma does not rely 
on a specific monomial form of summands there) that for any $\epsilon>0$,
\begin{equation}\label{A7}
\lim_{k\to\infty} \limsup_{N\to\infty}P[\sup_{0\le t\le T} |\eta_N^k(t)|\ge 
\epsilon]=0
\end{equation}
It is enough to show that  for fixed $k$ the limit theorem is  valid for 
$\xi_N^k(t)$ as $N\to\infty$. We can then let $k\to\infty$ similarly to the 
end of Section \ref{sec6}.
We denote the block sums by $Y_i=\sum_{j\in B^+(i)} X_j$ and observe that 
$Y_1, Y_2,...$ are
i.i.d. by the $\ell$-dependence and stationarity.  Then
\begin{equation}\label{A6}
\xi_N^k(t)=\frac{1}{b_N}\sum_{j: j(k+\ell-1)\le Nt} [Y_j-ka_N]+\frac{1}{b_N}
\sum_{i\in I(t)}[X_i-a_N]=\zeta^k_N(t)+R^k_N(t)
\end{equation}
where $I(t)$ is an incomplete block at the end.  If we want to show convergence of finite dimensional distributions to the L\'{e}vy process given by (\ref{A4}) we need to show that
$$
\lim_{k\to\infty}\lim_{N\to\infty} \frac{N}{k}E[f(\frac {X_1+\cdots+X_k }{b_N})]=\int_H f(x) M^\ast(dx).
$$
In order to complete the proof of finite dimensional convergence we need to 
check  only that  for each fixed  $t$ the second term in (\ref{A6}) is 
negligible in probability. Since $k$ is fixed and $N\to\infty$ this is obvious.  The tails behavior of $\frac{X_1+\cdots+X_k}{b_N}$ is given by the image of $M^k$ by the map $x_1+x_2+\cdots+x_k\to x$ of $H^k\to H$.
Since $M^k=\sum_A M^k_A$ and they are $0$ unless $sp(A)\le \ell$, averaging 
over translations of $A$, ignoring a few terms at the ends, produces as 
$k\to\infty$, $M^\ast$ for the L\'{e}vy measure. The effect of the gap is a 
factor of $\frac{k}{k+\ell-1}$
that multiplies $M^\ast$  which tends to $1$ as $k\to\infty$.   Any partial 
block consists  of a sum of at most $k$ terms of $\frac{X_i}{b_N}$ and is 
negligible for large $N$. So is $\frac{a_N}{b_N}$.
 
 Next, we need to verify the centering. We use the truncated mean 
 $\frac{x}{1+|x|^2}$ 
which then appears in the representation as a counter term in the integrand
\begin{equation}\label{A21}
\int\big[e^{i<\xi,x>}-1-\sum_j \frac{ i\,\xi_j\,x_j}{1+|x_j|^2}\big] 
M^{\widehat A}_A(dx)
\end{equation}
But when we push forward  the measure $M^{\widehat A}$ from $H^{\widehat A}
\to H$ by taking the sum $y=\sum_{i\in A} x_i$ we end up with
\begin{align}
\int\big[e^{i\,\xi\,y}&-1- \frac{ i\,\xi\,y}{1+|y|^2}\big] M^{r,
\{\sigma_i\}}(dy)\notag\\
\qquad&=\int \big[e^{i\,\xi\,y}-1-\frac{ i\,\xi\,\sum_i x_i}{1+
|\sum_i x_i|^2}\big] M^{\widehat A}_A(dx)\label{A22}
\end{align}
The difference between the two  counter terms  in (\ref{A21}) and (\ref{A22}) 
is $\gamma_{r,\{\sigma_i\}}$ and it adds up to
$$
\gamma=\sum_{r,\{\sigma_i\}}\gamma_{r,\{\sigma_i\}}
$$
which appears outside of the integral in (\ref{A8}) and does not influence
the limiting measure $M^*$.

Now, if the condition (\ref{A3}) is satisfied then convergence in the $J_1$ 
topology follows from Proposition \ref{prop5.1} obtained in \cite{TK} as a 
corollary of a more general result but we will still give an alternative 
direct proof below. Since convergence of finite dimensional distributions
was already obtained above it remains to establish tightness of the processes
 $\xi_N^k,\, N\geq 1$ and the conclusion of the proof is similar to Section 
 \ref{sec6} by letting $k\to\infty$. Observe though that the convergence of
  finite dimensional distributions was established above under the condition
 (\ref{A1}) while we claim that (\ref{A3}) already implies (\ref{A1}),
 and so we discuss this issue first.

Set $m=$dim$H$ and let $Z_i=(X_j,\, j\in B^+(i))$ be $mk$-dimensional random
vectors whose $m$-dimensional components are $X_j$'s and their sum amounts to
$Y_i$. The random vectors $Z_i,\, i=1,2,...$ are i.i.d. and in view of 
(\ref{A3}) we conclude by Lemma \ref{lem4.1} that they have regularly varying 
tails and, in particular, that (\ref{A1}) holds true. Thus by \cite{Rv} 
(or by Section 7.2 in \cite{Re}) the vector process
\[
\Psi^k_N(t)=\frac 1b_N\sum_{i:\, i(k+\ell-1)\leq Nt}\big (X_j-a_N,\, j\in 
B^+(i)\big)
\]
converges weakly in the $J_1$ topology as $N\to\infty$ to a vector L\' evy
process. Considering the linear map $(x_1,...,x_k)\longrightarrow x_1+\cdots
+x_k$ of $H^k$ to $H$ we see by Remark \ref{remA3} that 
\[
\zeta^k_N(t)=\frac 1{b_N}\sum_{i:\, i(k+\ell-)\leq Nt}(Y_i-ka_N)
\]
also converges weakly in the $J_1$ topology as $N\to\infty$ to a corresponding
L\' evy process which implies also tightness of the sequence of processes
$\{\zeta^k_N,\, N\geq 1\}$.

Now, we just need to make sure that  the summation that has  been carried out
 over blocks    $Y=X_1+\cdots+X_k$ will still allow us to derive tightness in
  the $J_1$ topology of the processes $\xi_N^k$ which amounts to boundedness
   and modulus of continuity estimates (see, for instance,
\cite{JS}, Theorem 3.21 in Ch.VI). The $J_1$ tightness of processes 
$\{\zeta^k_N,\, N\geq 1\}$ explained above yields estimates of $D[0,T]$ 
modulus of continuity of these processes in the following form.

For any $\epsilon>0,\eta>0$, there is a $\theta>0$  and a set  $\Delta(N,\theta,
 \epsilon)$ such that 
$$
\limsup_{N\to\infty} P[\Delta(N, \theta,\epsilon)]<\eta
$$
and on the complement $[\Delta(N,\theta,\epsilon)]^c$, given any $u,v$ with 
$|u-v|<\theta$, either
\begin{equation}\label{A11}
\sup_{\frac{Nu}{k_\ell}\le j\le\frac{Nv}{k_\ell}}|\frac{1}{b_N}\sum_{\frac{Nu}
{k_\ell}\le r \le j} (Y_r-ka_N)|<\epsilon
\end{equation}
or there is an integer $q$ such that $Nu\le q\le Nv$ and   both
\begin{equation}\label{A12}
\sup_{\frac{Nu}{k_\ell}\le j\le q-1}|\frac{1}{b_N}\sum_{\frac{Nu}{k_\ell}
\le r \le j} (Y_r-ka_N)|<\epsilon
\end{equation}
and
\begin{equation}\label{A13}
\sup_{q+1\le j\le \frac{Nv}{k_\ell}}|\frac{1}{b_N}\sum_{q+1\le r \le j} 
(Y_r-ka_N)|<\epsilon
\end{equation}
where for brevity we set $k_\ell=k+\ell-1$.
We need a similar estimate for $\xi^k_N(t)$ the  process of partial sums  
of $\{\frac{1}{b_N}[X_i-a_N]\}$.
From (\ref{A3}), we can assume that for any $\eta>0,\delta>0$, there is a 
$\theta>0$ such that 
$$
\limsup_{N\to\infty} P[F(N,\theta,\delta)]<\eta
$$
where 
$$
F(N,\theta,\delta)=\cup_{i,j:|i-j|\le N\theta}\{||X_i-a_N]\ge \delta b_N\ \&\  
|X_j-a_N|\ge \delta b_N\}.
$$
Suppose $Y=X_1+\cdots X_k$ is the sum over a block. On $[F_{N,\theta,
\delta}]^c$,
if some $|X_i-a_N|\ge \delta b_N$, then $|X_j-a_N|\le \delta b_N $ for 
$j\not=i$ and therefore
$$
|X_i-a_N|\le |Y-ka_N|+\sum_{j: j\not=i}  |X_j-a_N|\le |Y-ka_N|+(k-1)\delta b_N
$$
In particular, if $ |Y-ka_N|\le \epsilon b_N$ then $\sup_{1\le i\le k}|X_i-a_N|
\le ((k-1) \delta+\epsilon) b_N$. On the other hand, if $ |Y-ka_N|\ge \epsilon 
b_N$ then  $|X_i-a_N|\ge \frac{\epsilon}{k} b_N$  for some $i$. If $\delta\le 
\frac{\epsilon}{k}$ then $|X_i-a_N|\ge \delta b_N$ and on $[F(N,\theta,
\delta)]^c$, for $j\not=i$, $|X_j-a_N|\le \delta b_N$. 

Now we can estimate the $D[0,T]$  modulus of continuity of the $\xi^k_N$ 
process. Let $\theta $ be small enough and $|i-j|<N\theta$. Then there are 
blocks $B^+(r_i)$ and $B^+(r_j)$ to which $i$ and 
$j$ belong. We consider the rescaled partial sums of the process $Y_r$ in the 
interval  $r_i\le r  \le r_j$. 
We restrict ourselves  to the set $[\Delta(N,\theta,\epsilon)]^c$. Suppose the 
alternative (\ref{A11}) holds. In particular $|Y_r-ka_N|\le \epsilon b_N $ for 
every $r$ in the range. Any $X_q$ 
belonging to any of the blocks will satisfy $\frac{1}{b_N}|X_q-a_N|\le 
(\epsilon+(k-1)\delta)$. Therefore the analog of (\ref{A11}) holds and
\begin{equation}\label{A14}
\sup_{i \le q\le j}|\frac{1}{b_N}\sum_{i\le r \le q} (X_q-a_N)|<\epsilon+k 
(\epsilon+(k-1)\delta)=f_k(\epsilon,\delta)
\end{equation}
which goes to $0$ with $\epsilon$ and $\delta$.

Suppose the alternatives (\ref{A12}) and (\ref{A13}) hold. This provides us a 
 block $B^+(q)$. Suppose $\frac{1}{b_N}|Y_q-ka_N|<\epsilon$ we are in the 
 previous situation of (\ref{A11})
with $3\epsilon$ replacing $\epsilon$. If $\frac{1}{b_N}|Y_q-ka_N|>\epsilon$ 
 and  $\delta<\frac{\epsilon}{k}$, there is a $q\prime$ in the block with  
 $\frac{1}{b_N}|X^\prime_q-a_N|\ge \delta b_N$.  On $[F(N,\theta,\delta)]^c$, 
 for any $q^{\prime\prime}\not= q^\prime$ in $B^+(q)$, 
$\frac{1}{b_N}|X^{\prime\prime}_q-a_N|\le \delta b_N$, and combined with 
(\ref{A12}) and (\ref{A13}) this proved the modulus of continuity estimate for
 the process $\xi^k_N(\cdot)$ in the $J_1$ topology. Now the tightness follows 
 (see Theorem 3.21, Ch.VI in \cite{JS})
since we obtain also uniform boundedness in probability of processes $\xi^k_N$
 from the corresponding result for normalized sums $\zeta^k_N$ of independent 
 blocks together with the above estimates. 
 
 As mentioned above, letting $k\to\infty$ similarly to the end of Section 
 \ref{sec6} we obtain weak convergence in the $J_1$ topology of processes
 $\xi_N$ to a L\' evy process with the measure $M^*$. Finally, we observe
 that under the condition (\ref{A3}) the right hand side of (\ref{A4})
 contains only $r=1$, and so $M^*=M$, completing the proof.
 \end{proof}

\end{document}